\documentclass[aop,preprint]{imsart}

\RequirePackage{amsthm,amsmath,amsfonts,amssymb}
\RequirePackage[numbers]{natbib}

\startlocaldefs
\usepackage{graphicx}
\usepackage{enumitem}

\usepackage[displaymath, mathlines, pagewise]{lineno}
\numberwithin{equation}{section}
\numberwithin{figure}{section}

\theoremstyle{definition}

\newtheorem{theorem}{Theorem}
\numberwithin{theorem}{section}

\newtheorem{lemma}[theorem]{Lemma}

\newtheorem{remark}[theorem]{Remark}
\newtheorem{proposition}[theorem]{Proposition}

\newtheorem{question}[theorem]{Question}

\newtheorem{definition}{Definition}

\newcommand{\R}{\mathbb{R}}

\newcommand{\Z}{\mathbb{Z}}
\newcommand{\N}{\mathbb{N}}

\newcommand{\Rr}{\mathbb{R}^d}

\def\Isom{{\rm Isom}}

\def \eps {\epsilon}

\def \P {{\Bbb P}}
\def\Var{{\rm Var}}

\def \E {{\Bbb E}}

\def \_reg {\rightarrow_{\bf reg}}
\def \sep{cut}
\def\maxdeg/{\Delta}

\def\max{{\rm max}}
\def\min{{\rm min}}

\def\dist{{\rm dist}}

\def\calp{{\cal P}}
\def\calP{{\cal P}}

\def\dist{{\rm dist}}
\def\Vor{{\rm Vor}}

\def\F{{\cal F}}

\def\X{{X}}
\def\Aut{{\rm Aut}}

\def \eps {\epsilon}

\def \P {{\bf P}}

\def \E {{\bf E}}

\def \_reg {\rightarrow_{\bf reg}}
\def \sep{cut}
\def\maxdeg/{\Delta}

\def\max{{\rm max}}
\def\min{{\rm min}}

\def\dist{{\rm dist}}

\def\Pcal{{\cal P}}
\def\Q{{\cal Q}}

\def\dist{{\rm dist}}
\def\Vor{{\rm Vor}}

\def\F{{\cal F}}
\def\calq{{\cal Q}}

\def\X{{X}}
\def\G{{\cal G}}
\def\A{{\cal A}}

\def\cali{{\cal I}}
\def\Area{{\rm Area}}

\endlocaldefs

\begin{document}

\begin{frontmatter}
\title{A nonamenable ``factor'' of a Euclidean space}
%\title{A sample article title with some additional note\thanksref{T1}}
\runtitle{A nonamenable ``factor'' of a Euclidean space}
%\thankstext{T1}{A sample of additional note to the title.}

\begin{aug}
%%%%%%%%%%%%%%%%%%%%%%%%%%%%%%%%%%%%%%%%%%%%%%
%%Only one address is permitted per author. %%
%%Only division, organization and e-mail is %%
%%included in the address.                  %%
%%Additional information can be included in %%
%%the Acknowledgments section if necessary. %%
%%%%%%%%%%%%%%%%%%%%%%%%%%%%%%%%%%%%%%%%%%%%%%
\author[A]{\fnms{\'Ad\'am} \snm{Tim\'ar}\ead[label=e1]{madaramit@gmail.com}},
%%%%%%%%%%%%%%%%%%%%%%%%%%%%%%%%%%%%%%%%%%%%%%
%% Addresses                                %%
%%%%%%%%%%%%%%%%%%%%%%%%%%%%%%%%%%%%%%%%%%%%%%
\address[A]{Alfr\'ed R\'enyi Institute of Mathematics and \\
University of Iceland, \printead{e1}}
\end{aug}

\begin{abstract}
Answering a question of Benjamini, we present an isometry-invariant random partition of the Euclidean space $\R^d$, $d\geq 3$, into infinite connected indistinguishable pieces, such that the adjacency graph defined on the pieces is the 3-regular infinite tree. 
Along the way, it is proved that any finitely generated one-ended amenable Cayley graph can be 
represented in $\R^d$ as an isometry-invariant random partition of $\R^d$ to bounded polyhedra, and also as an isometry-invariant random partition of $\R^d$ to indistinguishable pieces.
A new technique is developed to prove indistinguishability for certain constructions, connecting this notion to factor of iid's. 
\end{abstract}



\end{frontmatter}
%%%%%%%%%%%%%%%%%%%%%%%%%%%%%%%%%%%%%%%%%%%%%%
%% Please use \tableofcontents for articles %%
%% with 50 pages and more                   %%
%%%%%%%%%%%%%%%%%%%%%%%%%%%%%%%%%%%%%%%%%%%%%%
%\tableofcontents

%%%%%%%%%%%%%%%%%%%%%%%%%%%%%%%%%%%%%%%%%%%%%%
%%%% Main text entry area:

%%%%%%%%%%%%%%%%%%%%%%%%%%%%%%%%%%%%%%%%%%%%%%
%% Supplementary Material, if any, should   %%
%% be provided in {supplement} environment  %%
%% with title inside \textbf{} and short    %%
%% description below.                       %%
%%%%%%%%%%%%%%%%%%%%%%%%%%%%%%%%%%%%%%%%%%%%%%
%\begin{supplement}
%\textbf{???}.
%???.
%\end{supplement}

\section{Introduction}\label{s.intro}

%Say that a subset of $\Rr$ is nice, if there exists a $c>0$ such that $P$ is the union of countably many, pairwise disjoint polyhedra $P$ %such that each $P$ contains some ball of radius $c$. Say that two nice subsets of $\Rr$ are adjacent, if the intersection of their closure 

\begin{definition}\label{maindef}{\bf (Tiling representation)}
Let $G$ be a finite or infinite graph. Say that the set $\calP$ is a {\it (locally finite) tiling of $\R^d$ that represents $G$
(or $\calP$ is a tiling representation of $G$, or $G$ is the adjacency graph of the tiling $\calP$)},
if the following hold.
\begin{enumerate}
\item Every element of $\calP$ is a connected open polytope (a {\it tile})
in $\Rr$. A polytope may be unbounded, with infinitely many hyperfaces.
\item The elements of $\calP$ are pairwise disjoint, the union of their closures is $\Rr$.
\item Every ball in $\R^d$ intersects finitely many elements of $\calP$.
\item Say that two elements of $\calP $ are adjacent if their closures share a $d-1$-dimensional face.
Then the graph defined on $\calP $ this way is isomorphic to $G$. 

We call the elements of $\calP$ {\it pieces} or {\it tiles} of $\calP$.
\end{enumerate}
\end{definition}

Representing a Cayley graph of a countable group $G$ as a {\it periodic tiling} of $\R^d$ is not possible for most $G$. A natural relaxation of periodicity is to take a {\it random tiling}, i.e., a probability measure on tilings, {\it whose distribution is invariant} under the isometries of $\R^d$. Instead of congruent tiles, one can ask for the probabilistic analogue of congruence, and require the tiles to be {\it indistinguishable}.

\begin{question}\label{qmain} (Itai Benjamini)
Is there an invariant random tiling representation of $T_3$ in $\R^3$ such that the tiles in 
this representation are indistinguishable?
\end{question}

Invariance is understood with regard to the isometries of $\R^3$, but we will look at other possible interpretations as well. 
By the {\it indistinguishability} of the tiles we mean the following. Let $\X$ be the set of all closed subsets of the Euclidean space $\Rr$, and consider the Hausdorff metric on it. Suppose that some $A\subset X$ is Borel measurable, and is closed under isometries of $\Rr$ (i.e., if a set is in $A$ then all its isometric copies are also in $A$). One can think of $A$ as the collection of subsets satisfying a certain measurable property: having some given congruence class, diameter at most $D$, some given lower/upper density, various topological properties... We say that the {\it pieces of a random partition of $\Rr$ are indistinguishable} if for any such $A$ either every piece of the partition is in $A$ almost surely, or none of them. It is easy to check that if there are bounded pieces with positive probability, then they are either all congruent, or they fail to be indistinguishable.

If the infinite 3-regular tree $T_3$ is embedded into $\R^d$ in an $\Aut (T_3)$-invariant way, then the expected number of vertices in a fixed cube is infinite; see Proposition \ref{nincsfa} for this folklore statement. 
(From now on, $\Aut (G)$ will denote the automorphism group of a given graph $G$. If $G$ is a diagram, i.e. it has colored and oriented edges, then $\Aut (G)$ stands for automorphisms that preserve the orientations and colors.) This implies, for example, that for an invariant point process of finite intensity, there is no way to define an invariant copy of $T_3$ on the configuration points as vertices.
%It follows from \cite{BK} that there exists no tiling as in Question \ref{qmain} if we want it to be in $\R^2$. 
It is not hard to prove that in $\R^2$ there exists no tiling as in Question \ref{qmain}, see Remark \ref{dim2}.
Given all these negative results, one may expect that a partition as in the question does not exist.
However, we prove that the answer to the question is positive. 
%Our construction works for any $\R^d$, $d\geq 3$; it is stated for $d=3$ to make the proof a bit simpler. 
We mention that in his original formulation, Benjamini did not require the tiles to be polyhedra, but any kind of pathwise connected domains. We included this condition in the definition of a representation by tiles because our construction works even with this extra requirement. 

\begin{theorem}\label{question}{\bf (Regular tree tiling in $\R^d$)}
For $d\geq 3$ there exists a random locally finite tiling of $\R^d$ that represents $T_3$, has indistinguishable pieces, and has a distribution that is invariant under the isometries of $\R^d$. The representation can also be viewed as an $\Aut (T_3)$-invariant map from $V(T_3)$ to the set of tiles.
\end{theorem}
A locally finite invariant tiling representation of $T_3$ must have tiles of infinite volume, as shown by Proposition \ref{nincsfa}. Proposition
\ref{counterexample} further shows that without the indistinguishability constraint there is a simple example.

The second part of Theorem \ref{question} claims that the decoration of $T_3$ with the tiles will be invariant under the automorphisms of $T_3$. In fact, we will first construct such a decoration, and then show that the corresponding tiling is invariant under the group of isometries $\Isom (\R^d)$ of $\R^d$ ($d\geq 3$). To put it in a slightly simplified way, an important issue will be the distinction between random maps from $T_3$ that are $\Aut (T_3)$-invariant, and maps whose image set in $\R^d$ is $\Isom (\R^d)$-invariant. Section \ref{s.duality} will address some related questions.

A major ingredient is the following:

\begin{theorem}\label{mackovilag}{\bf (Tiling representation of amenable graphs)}
Let $G$ be a locally finite unimodular transitive {\it amenable} one-ended graph or decorated graph. Then for $d\geq 3$ there is an $\Isom (\R^d)$-invariant locally finite random tiling of $\R^d$ that represents $G$, such that every tile is bounded. The representation can also be viewed as a map from $V(G)$ to the set of tiles, which is $\Aut (G)$-invariant and moreover is a factor of iid.
\end{theorem}

%It is easy to see that a tiling as in the theorem is not possible when $G$ is nonamenable. Hence the theorem can be viewed as a characterization of the Ornstein-Weiss theorem; see also the remark after Theorem \ref{OW}. 

See Definition \ref{fiid} for the definition of a factor of iid (fiid). 
We will prove the above theorem for the more general case of amenable unimodular random graphs $G$; see Theorem \ref{Gdomains}. The tiles in the theorem will not only be bounded polyhedra, but they will have only finitely many 0-faces (vertices), and hence finitely many faces of any dimension. The case of representing unimodular random {\it planar} graphs by invariant tilings in the {\it plane} has been investigated in a joint work with Benjamini, \cite{BT}, as a follow-up to the present work.

Returning to the initial question, to what extent are Cayley graphs $G$ representable as an invariant tiling of $\R^d$ with {\it indistinguishable} tiles, our method gives a positive answer not only for $T_3$ but also for all amenable groups.

\begin{theorem}\label{amenindist}{\bf (Amenable graphs by indistinguishable pieces)}
Let $G$ be a locally finite Cayley graph of an infinite {\it amenable} group. Then for $d\geq 3$ there is a random tiling of $\R^d$ with indistinguishable tiles, such that the adjacency graph of the tiling is $G$, and such that the tiling is invariant under the isometries of $\R^d$. The representation can also be viewed as an $\Aut (G)$-invariant map from $V(G)$ to the set of tiles.
\end{theorem}

In the remainder of this section, we sketch the proof of Theorem \ref{question} from Theorem \ref{mackovilag}.

\noindent {\bf Observation} (Damien Gaboriau) \hspace{0.1in} The usual Cayley graph of the Baumslag-Solitar group BS(1,2)$=\langle a,b|a^{-1}ba=b^2\rangle$ can be partitioned into connected pieces such that the adjecency graph between the pieces is $T_3$.

\vspace{0.1in} 
Choose the pieces for this partition to be the orbits of the generator $b$. We will call them {\it fibers}. We leave a formal proof that the adjacency graph on the fibers is $T_3$ to the interested reader, see Figure \ref{BS2} for an illustration.

\begin{figure}[h]
\vspace{0.2in}
\begin{center}
\includegraphics[keepaspectratio,scale=0.7]{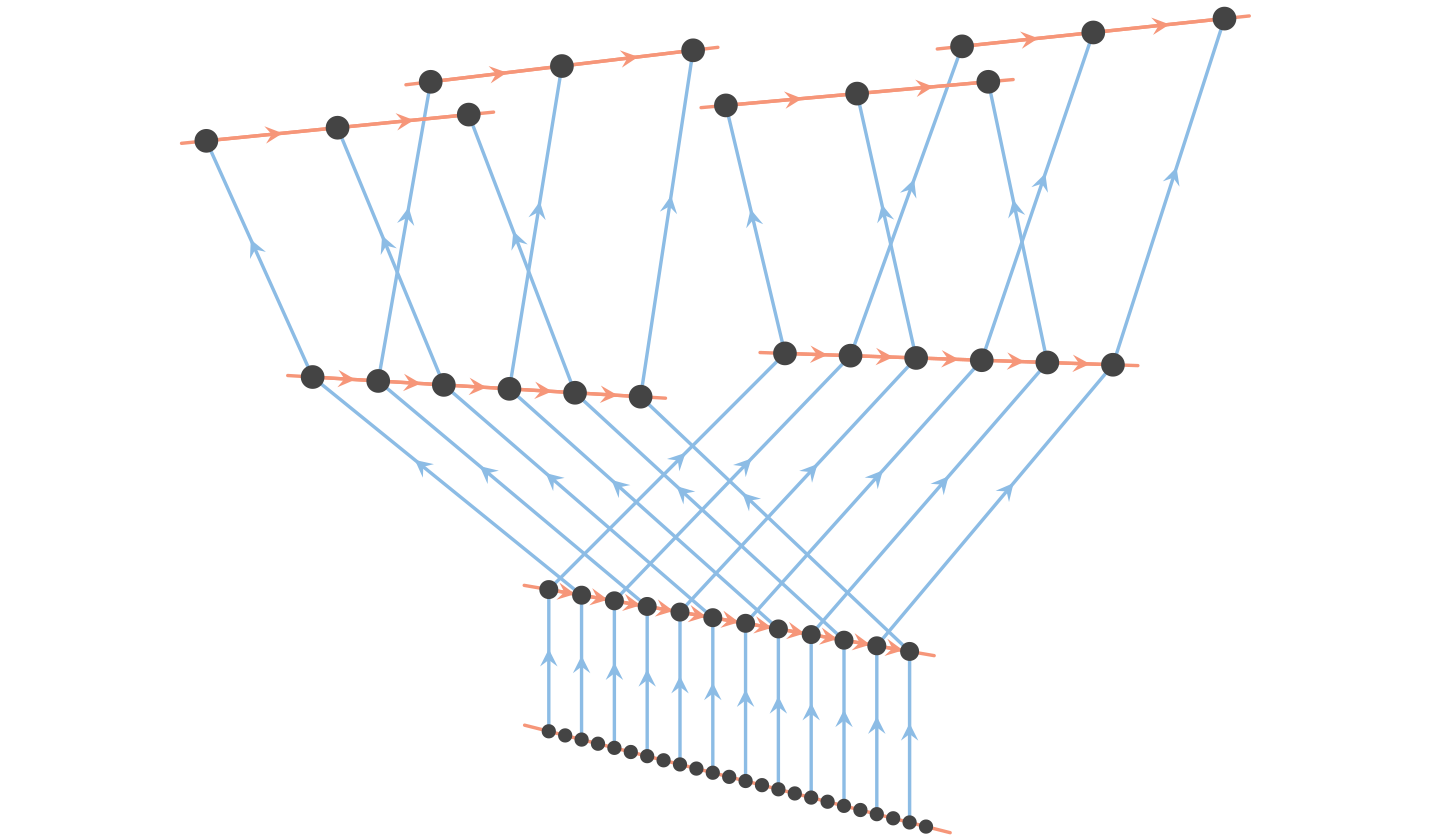}
\caption{Part of the Cayley graph of BS(1,2). Red lines (``fibers'') as partition classes provide a partition whose adjacency graph is $T_3$. (Image by Jens Bossaert.)}\label{BS2}
\end{center}
\end{figure}

The proof of Theorem \ref{question} is based on the following steps: 

\begin{enumerate}
\item Represent BS(1,2)$=:G$ in $\R^d$ as an isometry-invariant tiling, using Theorem \ref{mackovilag}.
\item Take the unions of tiles over each fiber in this representation (more precisely, the interior of the union of their closures), to obtain a representation of $T_3$.

\end{enumerate}

Much of the work will be in ensuring that the resulting tiles are in fact indistinguishable. Proving indistinguishability is usually highly non-trivial, see \cite{LS} for the case of the infinite components of Bernoulli percolation on Cayley graphs. A definition of indistinguishable {\it decorated connected components} of an automorphism-invariant random subgraph (percolation) of an underlying Cayley graph will be needed, which will be similar to our definition of indistinguishable tiles of a tiling, as above.
Components of a percolation on a Cayley graph turn into a unimodular random graph, when looking at the component of the origin, with the rest of the Cayley graph thought of as a (unimodular) decoration. Indistinguishability then transforms into ergodicity (extremality) of this unimodular probability measure. We will prove that in the above construction the
component of a fixed origin (as a unimodular decorated graph) is ergodic, instead of directly showing 
that the components of the percolation are indistinguishable. The advantage of addressing this property is that the usual definition of indistinguishability uses the automorphisms of the undelying graph
while
ergodicity of unimodular measures does not refer to that. We will have different graphs $G$ and $H$ on the same vertex set.
A subset of the vertex set induces a set of edges in $H$ and in $G$ as well. We will switch from viewing it as a subgraph of $H$, with some decoration from $G$ and the rest of $H$, to viewing it as a subgraph of $G$ with some decoration from $H$ and the rest of $G$. When doing so, the definition of indistinguishability that is independent of the underlying graph ($G$ or $H$) will be useful.
See Definitions \ref{extremal} and \ref{indist_graph}. 

Let us sketch the proof of our main Theorem \ref{question}. First apply Theorem \ref{mackovilag} to the amenable Cayley diagram BS(1,2).
The construction in Theorem \ref{mackovilag} is such that instead of directly partitioning $\R^d$ ($d\geq 3$) into tiles whose adjacency graph is $G$, we will map to each vertex of $G$ a tile in $\R^d$, together with a ``scenery'' (which is the rest of the tiling, from the viewpoint of this tile). This map will be a fiid from $G$. 
The usefulness of fiid constructions may be surprising in this setup, since the main question is not directly connected to locality or graph convergence. The reason that a big part of our construction needed to be fiid is because of our way of proving indistinguishability of the final tiles in the representation of $T_3$.
If we have a percolation on a Cayley graph or diagram $G$, and the percolation has infinite components ({\it pieces}) that are indistinguishable with their sceneries, and have any fiid decoration of these vertices, the resulting {\it decorated pieces} will also be indistinguishable. (For the definition of a Cayley diagram, see the paragraph before Lemma \ref{lemma_fiid} and see Definition \ref{extremal} for indistinguishability in the case of (decorated) graphs.)
This claim is proved in Lemma \ref{lemma_fiid}, which we call ``Decoration lemma''. Applying it to the  ($\Aut(G)$-invariant) partition into fibers of BS(1,2), which are trivially indistinguishable, we obtain that the
tiling that we assign to BS(1,2) as a decoration will produce indistinguishable unions of tiles over the fibers.
This is almost what we need, except for that with this viewpoint we constructed an $\Aut (G)$-invariant partition of $G$, with tiles of $\R^d$ assigned to the fibers as an $\Aut(G)$-invariant (fiid) decoration, 
%and such that the classes of the partition together with the decoration are indistinguishable,
but instead we want the tiling to be {\it invariant under the isometries of $\R^d$}. So the question is whether one can switch from the $\Aut (G)$-invariant object to an $\Isom (\R^d)$-invariant object. This will be guaranteed essentially by Lemma \ref{dual}, which we called ``Duality lemma''. Informally, the lemma says (in a somewhat more general context) the following. Suppose that there is a random drawing of some Cayley {\it diagram} $H$ on the vertex set $V(G)$, and the distribution of $G$ with this random drawing is unimodular. (Note that the drawn edges of $H$ are not required to be elements of $E(G)$.) Now, we can switch and view $H$ as the fixed graph, and $G$ as the random decoration on $H$. The lemma says that then this random copy of $G$ on $H$ will be $\Aut (H)$-invariant.
This duality is close to what we need. First, the $\Aut (G)$-invariant
decoration of $V(G)$ by tiles in $\R^d$ can be taken to be a unimodular random decoration by $\Z^d$ and some extra information that describes the tiles. The Duality lemma tells us that this extra information and $G$ on $\Z^d$ is $\Aut (\Z^d)$-invariant. This can then be turned to be $\Isom (\R^d)$-invariant by applying a uniform isometry from $\Isom (\R^d)/\Aut (\Z^d)$.
%Besides the present paper, the Duality lemma was later also used in \cite{BT}. 

We mention that the above argument remains true if we replace $\Isom (\R^d)$ by the group of translations at each occurance. Hence there exists a translation-invariant tiling representation for $T_3$ with tiles that are indistinguishable with regard to translation-invariant properties. (Note that the latter is stronger than being indistinguishable under $\Isom (\R^d)$-invariant properties.) Moreover, since the construction in our proof is also $\Isom (\R^d)$-invariant, it provides us with an $\Isom (\R^d)$-invariant tiling representation with tiles that are indistinguishable with regard to translations. 

While preparing the present manuscript, a conference version of a weaker result, using some different methods, was published in \cite{T2}. 

The paper is organized as follows. In the next section we provide the necessary definitions and the proof of the Decoration lemma. The Duality lemma (Lemma \ref{dual}) is presented in Section \ref{s.duality}, together with examples illustrating the need for the lemma and some questions inspired by the lemma. Theorem \ref{mackovilag} will follow from the more general Theorem \ref{Gdomains}, which claims that one-ended amenable unimodular random graphs have representations by invariant random tilings consisting of bounded tiles. (Note that here the tiles are not expected to be indistinguishable.)
This will be proved in
Section \ref{s.treetiling}. 
In that section, the special case when $G$ is a one-ended unimodular tree is proved first, and then it is extended to any one-ended amenable unimodular graph, using the fact that such graphs have one-ended fiid spanning trees, \cite{T3}.
For the proof that a one-ended unimodular tree $T$ can be represented by an invariant tiling, we will need a technical lemma, verified in
Section \ref{s.diadic}. This lemma will provide us with a fiid sequence of coarser and coarser partitions $\calp_n$ of $V(T)$ such that ``many of the'' parts in the partition are connected subgraphs of $T$ with $2^n$ points. With proper care, one can define on each such part a piece of the $d$-dimensional grid (as a fiid), and so that in the limit we get a copy of $\Z^d$ ($d\geq 3$) on $V(T)$. This grid can be extended to a tiling as desired, using the usual embedding of $\Z^d$ in $\R^d$. The importance of the connectedness of some pieces in $\calp_n$ will be coming from the fact that such pieces can be nicely represented by tiles within a cube (as shown on Figure \ref{farepi}). These nice representations will be defined in such a way that their limit is the representation of $T$ by a tiling of $\R^d$, as desired. Finally, Section \ref{s.indisttiles} presents the proof of the main theorem, whose sketch we have provided already. The Decoration lemma is applied therein, to ensure that the tiles that represent $T_3$ are indistinguishable. 
Here and in the proof of Theorem \ref{Gdomains}, one will need the Duality lemma to obtain $\Isom (\R^d)$-invariance of the construction from the fact that it is $\Aut (G)$-invariant.

\section{Definitions; the ``Decoration lemma''}

The next few definitions can be found in \cite{AL}, together with some equivalent characterizations.

\begin{definition}\label{unimodular}{\bf (Unimodular random graphs, MTP)}
Let $\G_*$ be the set of all finite-degree connected rooted (multi)graphs up to rooted isomorphism.
In notation, we do not distinguish between a rooted graph and the rooted isomorphism class that it represents.
Define a distance on $\G_*$ by $d ((G,o),(G',o')):=\min\{1/r\,:\, B_G(o,r)$ and $B_{G'}(o',r)$ are rooted isomorphic$\}$. Let $\G_{**}$  be the set of finite-degree connected rooted graphs with two distinguished vertices, and up to double-rooted isometries. Similarly to $\G_*$, a metric can be defined on $\G_{**}$.
Let $\mu$ be some probability distribution on $\G_*$. We say that $\mu$ defines a {\it unimodular random graph} if for every Borel function $f:\G_{**}\to\R^+$, the following is true
\begin{equation}
\E (\sum_y f(G,x,y))=\E (\sum_y f(G,y,x))
\label{eq:mtp}
\end{equation}
where $(G,x)$ is the random element of $\G_*$ of distribution $\mu$. The above equation is called the Mass Transport Principle (MTP).
\end{definition}

\begin{definition}\label{extremal}{\bf (Ergodic unimodular random graphs)}
Call a subset $\A\subset\G_*$ an {\it invariant property} if it is Borel measurable and closed under the change of root (that is, if $(G,o)\in\A$ and $x\in V(G)$, then $(G,x)\in\A$). 
A unimodular random graph $(\Gamma,o)$ is {\it extremal} or {\it ergodic} if for every such $\A$ either $(\Gamma,o) \in\A$ almost surely or $(\Gamma,o)\not\in\A$ almost surely.
\end{definition}

We first define {\it decorated graphs}. This will be equivalent to what is called {\it marked} graphs in \cite{AL}, but we adjust it to our setting. 

\begin{definition}\label{decorgraph}{\bf (The space of decorated graphs)}
Fix some complete separable metric space $X$. Consider rooted graphs $(G,o)$ together with some decoration of $(G,o)$, where a decoration means a set $V'$ of {\it extra vertices} and a set $E'$ of {\it extra edges or oriented edges on} $V(G)\cup V'$ added to $G$ so that
$(V\cup V', E\cup E')$ is connected, 
and some partial coloring $\chi$ of $V(G)\cup V'\cup E(G)\cup E'$ with elements of $X$. Denote such a graph by $(G,o;V',E',\chi)$, but we will often drop $\chi$ from the notation.
Two such decorated rooted graphs will be equivalent if there is a rooted isomorphism that maps them to each other, maps the extra vertices and extra edges into each other isomorphically, and preserving the coloring from $X$. We will extend the meaning of $\G_*$ as the space of decorated graphs, and as before, refer to elements of $\G_*$ through representatives of the equivalence classes.
Consider $(G,o)$ with decoration $V',E'$ as above and $(\bar G,\bar o)$ with decoration $\bar V',\bar E'$. Say that they are at distance at most $1/r$ if there is a rooted isomorphism that maps the $r$-neighborhood of $o$ in $(V(G)\cup V', E(G)\cup E')$ to the $r$-neighborhood of $\bar o$ in $(V(\bar G)\cup \bar V', E(\bar G)\cup \bar E')$, and in such a way that it isometrically maps the $r$-neighborhood of $o$ in $G$ to the $r$-neighborhood of $\bar o$ in $\bar G$, and in such a way that
the $X$-colors of the vertices and edges in this ball that are mapped to each other differ by at most $1/r$. (In particular, uncolored vertices and edges are mapped bijectively to uncolored vertices and edges.)
%Now, with $\Gd$ and the metric on it, we can define indistinguishability of the infinite components of some percolation with decoration just the same way we did for the non-decorated case. 
\end{definition}

The definition of ergodicity extends to {\it decorated} unimodular graphs without any change.
Let us mention that it is possible to have multiple decorations on a graph (or to further decorate a decorated graph), by extending the space $X$ in the natural way. When it is convenient to have several decorations, we list them all after the semicolon. One important example for us is when we have a connected component $\omega_o$ of the fixed vertex $o$ in some percolation $\omega$ on a graph $G$, and {\it we look at $(\omega_o,o)$ as a rooted graph decorated with $G$ and $\omega$}. Then we will further take some (factor of iid) decoration of $\omega_o$, which may also use information from the decoration $G,\omega$.

Let $\omega$ be some $\Aut (G)$-invariant percolation (random subgraph) of a Cayley graph (or diagram) $G$. Denote by $\omega_o$ the connected component of a fixed vertex $o$ in $\omega$. As just said, we may consider
$\omega$ and $G$ as a decoration of the random graph $(\omega_o,o)$. Referring to this as
 $(\omega_o, o; G, \omega)$, we will say that $(\omega_o,o)$ is decorated with {\it scenery}. The decorated graph $(\omega_o, o; G, \omega)$ is unimodular, see e.g. \cite{AL}.

\begin{definition}\label{indist_graph}{\bf (Indistinguishability; indistinguishability with scenery)}
Say that $\omega$ has {\it indistinguishable components} almost surely if $(\omega_o,o)$ is ergodic. Say that the components of $\omega$ are {\it indistinguishable with scenery} if $(\omega_o, o; G, \omega)$ is ergodic.
\end{definition}
The connection between ergodicity of an $(\omega_o,o)$ and indistinguishability of the components of $\omega$  in greater generality is investigated in \cite{M}.
A simple example of indistinguishable components that are not indistinguishable with scenery is shown in Remark \ref{sceneryindist}.

\begin{remark}\label{counterexample}{\bf (Tiling representation without indistinguishability)}
There exists a random isometry-invariant tiling of $\R^d$ that represents $T_3$ for $d\geq 2$ if we do not require the pieces to be indistinguishable. It has bounded partition classes of different scales (which implies right away that some tiles can be distinguished, using their sizes). See Figure \ref{fractal} for the intuitive picture of the tiling. For simplicity, we do the construction of a $T_5$-partition in $\R^2$. Consider 
a sequence of random vectors $(v_i)_{i=-\infty}^\infty$, such that $v_i\in [0,2^i]^2$ and $v_{i+1}-v_i\in 2^i\Z^2$. In other words, $v_0\in[0,1]^2$ is uniform, for $i<0$ define $v_i\equiv v_{i+1} $ mod $2^i$, and for $i>0 $ define $v_{i+1}=v_i+\eps_i$, where $\eps_i\in\{0,2^i\}^2$ is uniform.
Define
sets of the form $v_i+w+2^i (1/5,2/5)
$ with $w\in 2^i\Z^2$
and sets $v_i+w+2^i (3/5,4/5)
$ with $w\in 2^i\Z^2$. Let the collection of all such sets be ${\cal K}_i$. Finally, define ${\cal S}_i$ to be the collection of sets
$K\setminus \cup\{ \bar L \,:\, {L\in {\cal K}_{i-1}}\}$, as $K$ ranges over ${\cal K}_i$. Then $\cup_i {\cal S}_i$ is a translation-invariant $T_5$-partition.
By applying a random uniform element of the factor of $\Isom (\R^2)$ by its subgroup of translations, we can make it isometry-invariant.
\end{remark}

\begin{figure}[h]
\vspace{0.2in}
\begin{center}
\includegraphics[keepaspectratio,scale=0.7]{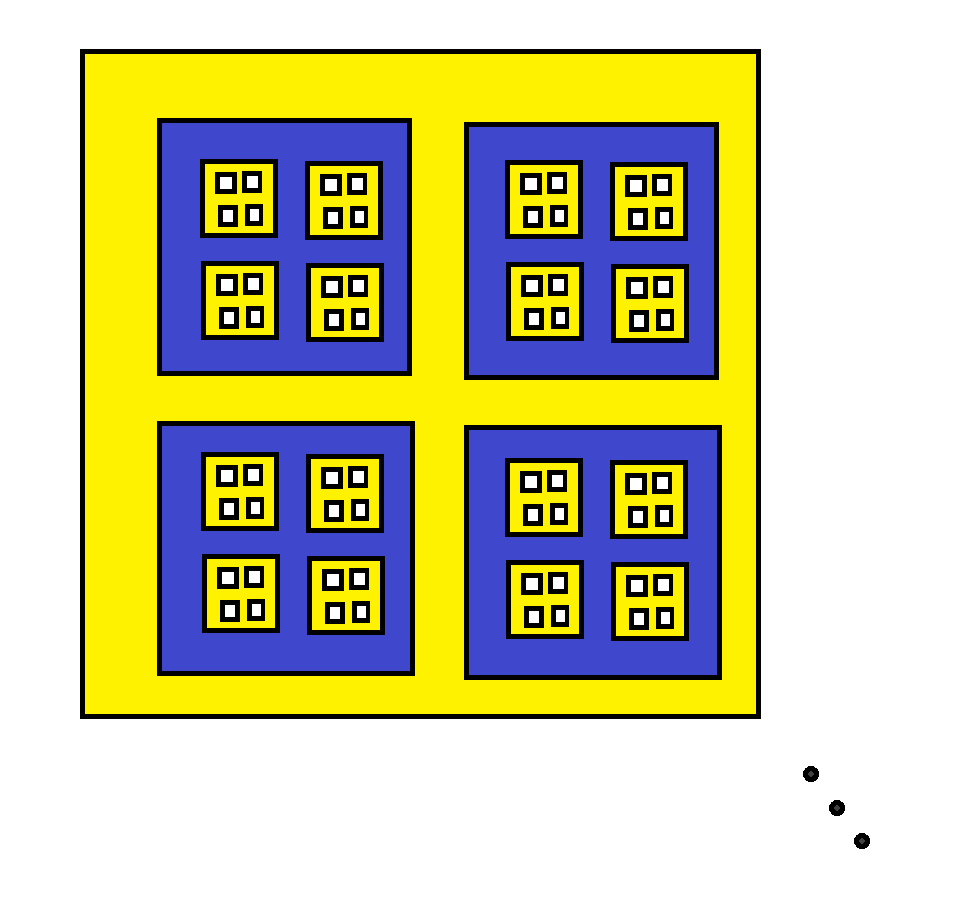}
\caption{Schematic representation of the 5-regular tree $T_5$ in $\R^d$ by a random tiling.
%, which can be made $\Isom (\R^d)$-invariant by applying a suitable random isometry to it
}\label{fractal}
\end{center}
\end{figure}

Let $G$ be an arbitrary graph or decorated graph. Put iid Lebesgue[0,1] random variables on its vertices. 
Consider some additional decoration on $V(G)$. Intuitively, we call it a {\it factor of iid} or {\it fiid} if
for any $\eps>0$ there is an $r$ such that one can tell the value of
the decoration of some $v\in V(G)$ from the labels in the $r$-neighborhood of $v$ in $G$ up to an arbitrarily small error.
That is, the decoration is determined locally, by the random labels. 

\begin{definition}\label{fiid}{\bf (Factor of iid)}
Let $X$ be some separable metric space.
%$G=(V(G),E(G))$ be a unimodular random graph, with iid Lebesgue[0,1] labels $\lambda\in [0,1]^{V(G)}$ on its vertices. Say that a measurable function $f:(\lambda, o)\mapsto x$ from ${\cal G}_*$ to $X$ is a {\it factor of iid} or {\it fiid}, if for any isomorphism $\gamma$ of $G$, $f(\gamma\lambda,\gamma x)=f(\lambda ,x)$.
A {\it factor map} is a Borel measurable function $f : {\cal G}_* \to X$.
Let G be a random graph, let 
$\lambda : V (G)\to [0, 1]$ be iid Lebesgue[0,1] {\it labels}
on its vertices, and let $G(\lambda)$ be the random labeled graph given by
the labels $\lambda$. The collection of random variables 
$\bigl(f((G(\lambda),v))\bigr)_{v\in V(G)}$
is
called a {\it factor of iid (fiid)} process if $f$ is a factor map.
\end{definition}

%See e.g. \cite{LN} for more details.  
Throughout the paper we reserve the term {\it label} for the iid Lebesgue labels in the above definition, as opposed to other {\it decorations} or {\it colorings} of the graph.

A uniform number $c$ from $[0,1]$ can be used to define two independent uniform numbers $c_1,c_2$ from $[0,1]$. Namely, if $c=.\xi_1\xi_2\xi_3\ldots$ is the binary expansion of $c$ (so $\xi_i\in\{0,1\}$), then $c_1:=.\xi_1\xi_3\xi_5\ldots$ and $c_2:=.\xi_2\xi_4\xi_6\ldots$ are iid uniform. We can similarly split $c_2$, and continue ad infinitum. 
Therefore, 
in fiid constructions, we will 
assume that we have infinitely many iid
random $[0,1]$-labels on each vertex at our disposal, all independent from the others. 
Our construction of fiid maps will be via some local algorithm, and for convenience we often do not formalize how to turn it into a factor of iid rule. But it is always possible to turn such local rules into fiid: any additional randomness that is needed locally, can be extracted from the iid labels on the vertices, and one could fix a rule for this extraction that is applied for every vertex. E.g., we will
make local choices from finite sets without specifying the particular rule that will be used (e.g., choose the vertex of the set whose label is the smallest).

Given some group $G$ and a finite set of generators $S$, a {\it Cayley diagram} is a graph on vertex set $G$, and an oriented edge from $x\in G$ to $y\in G$ if $xg=y$, $g\in S$, in which case we color this edge by $g$. 
A Cayley graph is constructed from a Cayley diagram if we forget about the orientations and the colors of the edges.
We say that two Cayley diagrams are isomorphic if there is a graph isomorphism between them that also preserves the orientations and colors of the edges. By a slight abuse of terminology, but without ambiguity, we will use notation $\Z^d$ both for the group, the Cayley graph, and the Cayley diagram with respect to the standard generators. It will always be clear from the context, which one is understood.

For completeness, we provide the definition of amenability of unimodular random graphs. However, we will not need it, only a characterizing property of one-ended amenable unimodular random graphs: namely, that they contain a jointly unimodular one-ended random spanning tree. The notion of amenability was extended from transitive (unimodular) graphs to unimodular random graphs by Aldous and Lyons in \cite{AL}, and they gave several equivalents; see also Section 2 of \cite{T3} for the definitions. For a graph $H$ and vertex $x\in V(H)$, denote by $\deg (x,H)$ the degree of $x$ in $H$. 

\begin{definition}\label{amendef}{\bf (Amenability)}
A unimodular random graph $(G,o)$ is {\it amenable} if for any $\eps>0$ there is an $\omega\subset G$ such that $(G,o;\omega)$ is unimodular, every component of $\omega$ is finite, and $\deg (o,G)-\deg(o,\omega)<\eps$.
\end{definition}

\begin{lemma}{\bf (Decoration lemma)}\label{lemma_fiid}
Let $G$ be a (locally finite) Cayley graph or Cayley diagram, and $\omega$ be an $\Aut(G)$-invariant percolation. Let $\delta$ be some factor of iid decoration on $G,\omega$, where the iid labels $\lambda:V(G)\to [0,1]$ used for the factor are independent from the percolation. Suppose that all the components (pieces) of $\omega $ are infinite and indistinguishable with scenery. Then they are also indistinguishable with scenery in the full decoration of $\omega$ with the fiid.
\end{lemma}
%It follows that for any fiid decoration of the fibers of B(1,2) (which are trivially indistinguishable subgraphs), the decorated fibers are indistinguishable.

\begin{remark}\label{sceneryindist}
The lemma is not true if we only assume indistinguishability of the components: the different sceneries can be used to define different decorations, which are trivially fiid. For example, consider $G=\Z^2$, and let $i\in \{1,\ldots,5\}$ be uniform random, $t\in\{0,1\}$ independent uniform random, and $\omega$ consist of all horizontal lines of vertical coordinate $i$, $i+1$ or $i+3$ mod 5 if $t=0$, and
all vertical lines of horizontal coordinate $i$, $i+1$ or $i+3$ mod 5 if $t=1$. The components are indistinguishable (but they are not indistinguishable with scenery). Color with red the components that are at distance 1 from some other component, and color all other components with green. This decoration is a trivial fiid (which does not depend on the labels), and results in distinguishable decorated components.
\end{remark}

%The lemma is true more generally: ergodicity of a unimodular random (decorated) graph is preserved by additional fiid decorations. To reduce the amount of technicalities needed, we decided to include only the simpler version that we need here. 
The proof is far from surprising, only the necessary notation makes it cumbersome.

\def\Bomega{{B^{\Omega}}}
\def\Blambda{{B^{\Omega,\Lambda}}}
\def\calf{{\cal F}}
\def\calflambda{{\cal F}^{\Lambda}}
\def\cali{{\cal I}}
\def\calilambda{{\cal I}^{\Lambda}}
\def\Alambda{{A^{\Lambda}}}

\begin{proof}
Fix a vertex $o$ and denote by $\omega_o$ its component in $\omega$.
Proving by contradiction, suppose that the claim is false.
By Definition \label{extremal} this means that there is an invariant property that the unimodular decorated random graph (with scenery) $(\omega_o,o; G,\omega,(\delta (\lambda,v))_{v\in V(G)})$ satisfies with probability strictly between 0 and 1.
Applying $\delta^{-1}$, we get an
invariant property of probability strictly between 0 and 1. So, to get a contradiction, it is enough to prove that 
%$(G,\omega)$ decorated with the iid labels $\lambda$ on $V(G)$ has indistinguishable decorated $\omega$-components, i.e., that
$(\omega_o,o; G,\omega,(\lambda,v)_{v\in V(G)})$ is ergodic. 
%the $\omega_o$ is ergodic as a unimodular graph with decoration $G$, $\omega$ and $\lambda$.

For a vertex $x\in V(G)$, let $B(x,R)$ be the rooted ball of radius $R$ around $x$ in $G$. We think about $B(x,R)$ as a graph, but keep track of the root.
From now on, let $\calf$ be the $\sigma$-algebra for the percolation $\omega$ on rooted graph $(G,o)$, and $\calflambda$ be the product $\sigma$-algebra of $\calf$ and that of the $\lambda$-labellings.
%Fix a vertex $o$.
Starting from $o$, run delayed simple random walk on $\omega$, which is defined as follows. If we are in a vertex $x$, choose uniformly a $G$-neighbor $y$ of $x$. If $y$ is a neighbor of $x$ in $\omega$, then move to $y$, otherwise stay in $x$. 
Let $X_n$ be the $n$'th step of this walk; so $X_0=o$. One can view this random walk as a process on rooted graphs decorated with $\omega$ (with root in the vertex where the walker is), and also as a process on rooted graphs decorated with $\omega$ and $\lambda$. 
Both processes are stationary with respect to the delayed random walk (Theorem 4.1 in \cite{AL}). 

Let $\cali$ (respectively, $\cali^\Lambda$) be the invariant $\sigma$-field of $\calf$ (respectively, $\calf^\Lambda$) with respect to the transition operator of the delayed simple random walk.
%, with probability space given by the percolation $\omega$. 
Ergodicity of $(\omega_o,o; G,\omega)$ is equivalent to saying that for every
$A\in \calf$, $\E ({\bf 1}(A)|\cali)=\P(A)$ almost everywhere. Our goal is to show that for every $\Alambda\in \calflambda$, $\E ({\bf 1}(\Alambda)|\calilambda)=\P(\Alambda)$ almost everywhere. It is enough to show this for every cylinder event $\Alambda$. So let us assume that $\Alambda$ is determined by the restriction of $\omega$ and $\lambda$ to $B(o,R)$, with some $R>0$. Denote by $\Bomega (x)$ the pair $(B(x,R), \omega|_{B(x,R)})$.
Let $\Blambda (x)$ be $\Bomega(x)$ together with the labels from $\lambda$, i.e., $(B(x,R), \omega|_{B(x,R)}, \lambda|_{B(x,R)})=(\Bomega (x), \lambda|_{B(x,R)})$. Whether an $(\omega,\lambda)$ configuration on $(G,o)$ is in $\Alambda$ is determined by the configuration in $B(o,R)$, that is, by $\Blambda (o)$. Therefore, we can say $\Blambda (o)\in\Alambda$ without ambiguity (but by a slight abuse of terminology).
%, as an equivalent to $(\omega,\lambda)\in\Alambda)$ on $(G,o)$.

By Birkhoff's ergodic theorem we have
\begin{equation}
\frac{1}{n}\sum_{i=1}^n {\bf 1}({\Blambda (X_i)\in \Alambda})\to \E \big({\bf 1}({\Blambda (X_i)\in \Alambda})\big|\calilambda\big)
%=\P (B(o,R)\sym \H_k),
\label{birkhoff}
\end{equation}
almost surely, and our goal is to prove that the right hand side is $\P(\Alambda)$ almost surely. In order to do that, first we will show that the expectation of the left hand side converges to $\P(\Alambda)$, and then prove that the second moment tends to 0.
We can assume that $\Alambda$ consists of elements where $\Bomega$ is equal to some fixed configuration $\hat\Bomega$. Then we obtain the result for more general $\Alambda$ by taking disjoint unions. Let $\Lambda_0$ be the set of $\lambda$-configurations $\lambda|_{B(o,R)}$ on the $R$-ball of $G$, such that $(\hat\Bomega (o), \lambda|_{B(o,R)})=\Blambda (o)$ is in $\Alambda$. 
%We know that $\Alambda=\{(\hat\Bomega, \lambda):\, \lambda\in\Lambda_0\}$.
By the independence of $\omega$ and $\lambda$, we have 
\begin{equation}
\P(\Alambda)=\P(\hat\Bomega)\P(\Lambda_0)
\label{egyszeru}
\end{equation}
For simplicity, introduce $\xi_i={\bf 1}({\Blambda (X_i)\in \Alambda})$.

Now, $\E(\sum_{i=1}^n \xi_i)=\E \bigl(\sum_{i=1}^n {\bf 1}({\Bomega(X_i)\cong\hat\Bomega})
\P(\lambda|_{B(X_i,R)}\in \Lambda_0)\bigr)$. (Here probabilies and expectations are understood jointly with respect to the unimodular measure, the labelling and the random walk.)
By the stationarity of the delayed random walk, $\P(\lambda|_{B(X_i,R)}\in \Lambda_0)=\P(\lambda|_{B(X_j,R)}\in \Lambda_0)=\P(\Lambda_0)$ for every $i,j$.
Denote $N(n):=\bigl|\{i\in\{1,\ldots, n\}: \Bomega (X_i)\cong\hat\Bomega
\}\bigr|$. 
Using the fact that the $\lambda$-labels and the random walk are independent, we obtain
\begin{equation}
\E(\sum_{i=1}^n \xi_i)=\E \bigl(\sum_{i=1}^n {\bf 1}({\Bomega (X_i)\cong\hat\Bomega})
\P(\lambda|_{B(X_i,R)}\in \Lambda_0)\bigr)=\E( N(n))\P(\Lambda_0).
\label{subidubi}
\end{equation}
Apply the ergodic theorem to our random walk with only the $\omega$ environment and forgetting about the $\lambda$. By assumption, the $\omega$-environment is ergodic, therefore $\frac{N(n)}{n}=\frac{1}{n}\sum_{i=1}^n {\bf 1}({\Bomega (X_i)\cong\hat\Bomega})\to \E\big({\bf 1}({\Bomega (o)\cong\hat\Bomega})\big|\cali\big)=\P(\hat \Bomega)$ almost surely. Thus \eqref{subidubi} can be rewritten as
\begin{equation}
\frac{1}{n}\E(\sum_{i=1}^n \xi_i)
\to \P(\hat\Bomega)\P(\Lambda_0).
%=\P (B(o,R)\sym \H_k),
\label{subidubi2}
\end{equation}
If $\dist_G (x,y)>2R$, then $\lambda|_{B(x,R)}$
is independent from $\lambda|_{B(y,R)}$, and hence 
${\bf 1}({\Blambda (x)\in \Alambda})$ and ${\bf 1}({\Blambda (y)\in \Alambda})$ are independent.
Then there is a constant $c$ such that 
for any $x\in V(G)$, the expectation of $|\{i\in\{1,\ldots,n\}, X_i=x
\}|$ is less than 
$c\sqrt{n}$ (by a standard argument such as the one after Theorem 8.2 in \cite{P}, which works for any infinite graph).
All this put together gives a second moment bound
$\Var (\frac{1}{n}\sum_{i=1}^n \xi_i)
\leq C n^{-1/2}$.

Let us summarize what we have seen about $\frac{1}{n}\sum_{i=1}^n \xi_i$. First, it is almost surely convergent, by \eqref{birkhoff}. Secondly, the second moment of this sequence tends to 0, so the limit random variable is almost surely a constant. Then, by \eqref{subidubi2} and \eqref{egyszeru} this constant is $\P(\Alambda)$. We conclude that
$$\lim _{n\to\infty}\frac{1}{n}\sum_{i=1}^n \xi_i=\P(\Alambda).
$$
Comparing this to \eqref{birkhoff}, we obtain that the labelled graph has a trivial invariant $\sigma$-field $\calilambda$. This is what we wanted to prove.
\end{proof}

The proof of the previous lemma inspired the following question, which seems to be open, somewhat surprisingly.

\begin{question}
Let $(G,o)$ be an ergodic unimodular random graph of bounded degrees, and $(X_n)$ be lazy random walk started from $X_0=o$, with laziness set up so that $(X_n)$ is stationary for every $n$ and such that $\P(X_0=X_1)>0$. Does the conditional distribution of $(G,X_t)$ given $(G,X_0)$ converge to the initial unimodular measure almost surely?
\end{question}
We mention that there is no straightforward way to apply the (Birkhoff) ergodic theorem for the averages given by the lazy random walk, since there is a simple construction of a measure preserving system and an $L^1$ function for which these averages do not converge, as noted to us by G\'abor Pete.

\section{Duality}\label{s.duality}
To motivate the present section, we point at the fact that there exists 
an $\Isom (\R^d)$-invariant embedded copy of $T_3$ in $\R^d$ that is not $\Aut (T_3)$-invariant (as a map from $T_3$ to this copy). See the next proposition (stated for $T_5$ for convenience) for the proof. This observation highlights that there are two possible interpretations of having an ``invariant'' copy of one space in the other,
and these two may not always hold at the same time.
Our constructions for Theorems \ref{question} and \ref{mackovilag} will, however, work in both senses of invariance, thanks to the correspondance established in this section.

\begin{proposition}\label{szezonesfazon}{\bf ($\Isom (\R^d)$-invariance vs $\Aut(T)$-invariance)}
Let $d\geq 2$.
There exists an $\Isom (\R^d)$-invariant random embedded copy of the 5-regular tree $T_5$ in $\R^d$ that does not arise as the image set of any $\Aut (T_5)$-invariant random map from $T_5$ to $\R^d$.
\end{proposition}
We mention that the random copy of $T_5$ in the claim is such that the vertices form a point process of infinite intensity. If one requires finite intensity, there is no example as in Proposition \ref{szezonesfazon}; see Proposition \ref{nincsfa}.

\begin{proof}
Choose a uniform point in each of the tiles in the construction of Remark \ref{counterexample}, and connect two by a straight line if and only if their tiles are adjacent. One almost surely gets an $\Isom(\R^d)$-invariant embedding of $T_5$ into $\R^d$ if $d\geq 3$. (If $d=2$, one can define some broken line segment between the two points in adjacent tiles so that no two such segments intersect. We leave the details to the interested reader.) On the other hand, suppose that $\phi$ were an $\Aut (T_5)$-invariant random map from $T_5$ to $\R^d$ whose image set were the above-constructed random embedded copy of $T_5$. Then for every vertex $v$ one could uniquely define a ``parent'' $w$ as the neighbor whose tile separates the tile of $v$ from infinity. The rule to define the parent is $\Aut (T_5)$-invariant. But then every vertex would be the parent of 4 other vertices, and would have a single parent, giving a MTP contradiction.
\end{proof}

A group of automorphisms is {\it regular} if it is transitive and the stabilizer of any point is trivial. The automorphism group of a Cayley diagram is always regular. Let $(\bar G; \bar \chi)$ be a fixed decorated graph whose automorphism group is regular, and let $(G,o;\chi, \chi')$ be some ergodic unimodular random graph (with decorations $\chi$ and $\chi'$). Fix a vertex $\bar o$ of $\bar G$.
Suppose that the decorated graph $(G;\chi)$ is almost surely isomorphic to $(\bar G; \bar \chi)$. By regularity, there is a unique isomorphism $\rho$ that maps $(G;\chi)$
to $(\bar G; \bar \chi)$ and such that $\rho (o)=\bar o$. 
Call it the {\it derooting map}. 
%By a slight abuse of notation, we denote by $\rho((G;\chi))$ the image set of the bijection from $(G;\chi)$ (which is just $(\bar G; \bar \chi)$), and similarly, by $\rho((G;\chi,\chi'))$ the image set of $(G;\chi,\chi')$, which is $(\bar G; \bar \chi)$ together with a random decoration given by $\rho$ mapping $\chi'$.

\begin{lemma}\label{unimod2invar}{\bf (From unimodular to invariant)}
Let $(G,o;\chi, \chi')$ be an ergodic unimodular random graph (with decorations $\chi$ and $\chi'$),
and suppose that $(G;\chi)$ is almost surely isomorphic to some fixed decorated graph $(\bar G; \bar \chi)$ whose automorphism group is regular. Then the derooting map takes $(G,o;\chi, \chi')$ to an $\Aut (\bar G)$-invariant random decorated graph. 
Less formally: if $(G,o;\chi, \chi')$ is unimodular and $\Aut((G; \chi))$ is regular then $(G;\chi, \chi')$ is $\Aut((G; \chi))$-invariant.
\end{lemma}

\begin{proof}
We will prove using the less formal language, referring directly to $(G;\chi)$, without explicit involvement of $(\bar G;\bar \chi)$.
Consider an arbitrary automorphism $\gamma$ of 
%$(\bar G;\bar\chi)=\rho((G;\chi))$. 
$(G;\chi)$.
Pick an arbitrary event $A$ determined by the random decorated graph $(G;\chi, \chi')$.
Define the following mass transport. Let every vertex $x$ of $G$ send mass 1 to vertex $\gamma x$ if the event $A$ holds. Then the expected mass sent out is $\P(A)$. By the unimodularity of $(G,o;\chi, \chi')$ and the MTP, this is the same as the expected mass received, which is $\P (\gamma^{-1}A)$. Since this is true for any $A$ and automorphism $\gamma$, we obtain the invariance claimed.
\end{proof}

\begin{lemma}\label{dual}{\bf (Duality lemma)}
%Let $(G,o)$ and $(H,o')$ be unimodular random graphs. Suppose that there is a unimodular decoration of $G$ by $H$, meaning that there is 
Consider some unimodular random decorated graph 
\\$((V,E),o; F,\chi)$, where $(V,E)$ and $(V,F)$ are connected graphs,
and $\chi:V\cup E\cup F\to {\cal X}$ is some arbitrary further decoration, given as a map to some metric space ${\cal X}$. Then the decorated graph $((V,F),o;E,\chi)$ is unimodular. If, furthermore, $((V,F);\chi|_{F})$ is a deterministic graph with a regular automorphism group, then $((V,F);E,\chi)$ is invariant under the automorphisms of $((V,F);\chi|_{F})$.
\end{lemma}
%Note that having a regular automorphism group is a natural assumption to remain well-defined: if the stabilizer of $o$ were not trivial, it would be unclear how to apply an automorphism to $((V,F),o;E,\chi)\in \Gstar$, which is only defined up to rooted isomorphisms.

\begin{proof}
%Let us hide $\chi$ in our notation, for the sake of simplicity.
%Denote by $\mu$ the probability measure that defines the unimodular random decorated graph $((V,E),o; F)$. Denote by $\iota:((V,E),o; F)\mapsto ((V,F),o;E)$ the map defined on the support of $\mu$ that switches the roles of the edge set $E$ and decoration $F$. Note that $\iota$ is injective: if $((V,F),o;E)$ and $((V,F'),o;E')$ are rooted isomorphic, then the same bijection from $V$ to itself shows that $((V,E),o;F)$ and $((V,E'),o;F')$ are also rooted isomorphic.
%We have to show that the push-forward $\mu'$ of $\mu$ by $\iota$ is also unimodular, that is, it satisfies Definition \ref{unimodular}. Consider an arbitrary Borel function $g:\G_{**}\to\R^+$. On every decorated rooted graph $(G,o)$ in the support of $\mu$, define $f(G,o,x):=g(\iota^{-1}((G,o)),x)$ from $\G_{**} $ to $\R^+$. Injectivity of $\iota$ implies that $f$ is well-defined. On the other hand, $f$ is Borel, because $\iota$ is a homeomorphism of $\G_{*} $. Thus \eqref{eq:mtp} holds for $f$ and $\mu$. But the respective sides remain the same if we replace $f$ by $g$ and $\mu$ by $\mu'$. We conclude that the MTP holds for $\mu'$.
The first half of the claim is trivial: both the assumption and the conclusion are the same as the joint unimodularity of $((V, E\cup F),o;F,\chi)$ (or, equivalently, $((V, E\cup F),o;E,\chi)$), by definition.

The second part of the claim follows directly from Lemma \ref{unimod2invar}.
\end{proof}

The next proposition is an application of the Duality lemma, and it complements Proposition \ref{szezonesfazon}. It will be needed later, and we have not found it in the literature, so we include an outline of the proof. One may think that a Burton-Keane type of argument would show the claim right away, but in order to do so, one would have to move between the $\Isom (\R^d)$-invariant and $\Aut (T_3)$-invariant worlds, which is, as Proposition \ref{szezonesfazon} indicates, not a triviality.
By a ``random copy'' of $T_3$ in $\R^d$ we mean that there is some ($\Isom (\R^d)$-invariant) point process $V$ in $\R^d$ and a graph defined on $V$ that is isomorphic to $T_3$. The {\it intensity} of an invariant point process is the expected number of points in a unit cube.

\begin{proposition}\label{nincsfa}
Let $T$ be an $\Isom (\R^d)$-invariant random copy of $T_3$ in $\R^d$, $d\geq 1$. Then $V(T)$ has infinite intensity in $\R^d$. 
\end{proposition}

\begin{proof}
Suppose to the contrary, that $V(T)$ has finite intensity $c$.
Let $T$ be as in the claim. 
Consider $Z$ to be a copy of the Cayley diagram of $c^{-d}\Z^d$ in $\R^d$, moved by a uniform element of $\Isom (c^{-d}\R^d)/\Aut (c^{-d}\Z^d)$, and independent from $T$. Then the joint distribution of $(T,Z)$ is invariant under $\Isom(\R^d)$. The two point processes $V(T)$ and $V(Z)$ have the same intensities, hence there exists an invariant perfect matching $m$ between them almost surely (e.g. consider the stable matching as in \cite{HHP}, using the easy fact that $(T,Z)$ has only trivial symmetries almost surely). One can use $m$ to define a 3-regular tree $T_Z$ on $V(Z)$, by adding a (new) edge between two vertices if their $m$-pairs are adjacent. This tree on the copy $Z$ of $\Z^d$ is $\Aut (\Z^d)$-invariant, by the $\Isom (\R^d)$ invariance of our construction. One can randomly invariantly 3-color the edges of $T_Z$ (and trivially orient them in both directions) to get a Cayley diagram of the 3-fold free product of the 2-element group with itself. Do this independently from all other randomness, and call this resulting Cayley diagram $F$. Then we obtain an $\Aut (\Z^d)$-invariant copy of $F$ on $\Z^d$. By Lemma \ref{dual}, this gives rise to an $\Aut (F)$-invariant copy of $\Z^d$ on $V(F)$. Let $P_i$ be a uniform random partition of $\Z^d$ into $2^i$ times $2^i$ cubes. Then the $P_i$ is also $\Aut (F)$-invariant. Let $F_i$ be the subgraph of $F$ consisting of edges whose endpoints are in the same piece of $P_i$. 
As $i$ tends to infinity, $\deg_{F_i}(x)\to 3$ for a fixed vertex $x$. But as soon as $\deg_{F_i}(x)> 2$, the forest $F_i$ has some infinite component, contradicting the fact that $P_i$ consists of only finite parts. This final contradiction shows that $V(T)$ cannot have finite intensity.
\end{proof}

The rest of this section is not needed for the results stated in the Introduction. We take a detour to some questions that were inspired by Lemma \ref{dual} and seem to be interesting on their own rights.

In our lemma a coupling between the two unimodular (decorated) graphs $(V,E)$ and $(V,F)$ {\it is given a priori}. The claim that the roles of ``underlying graph'' and ``decorating graph'' can be switched was trivial in this case. 
It is easy to check that the relationship ``$G$ has a unimodular decoration by $H$'' is transitive, hence it is an equivalence relation. (Note, however, that one has 
to choose the coupling instead of using an a priori given one. It is easy to construct an example where
$(G_1,o_1)$, $(G_2,o_2)$ and $(G_3,o_3)$ are unimodular random graphs, and there is a random rooted decorated graph $((V,E_1\cup E_2\cup E_3),o; E_1,E_2)$ such that $((V,E_i),o)$ has the same distribution as $(G_i,o_i)$ for every $1\leq i\leq 3$, and such that
$((V,E_i\cup E_{i+1}), o;E_i)$ is unimodular for $i=1,2$, but $((V,E_1\cup E_{3}), o;E_1)$ is not unimodular.) Call this relation between two unimodular (decorated) random graphs {\it decoration-equivalence}. A very similar notion, {\it coupling equivalence} was defined independently in \cite{AHNR}, where the following theorem was also proved (using similar arguments).

\begin{theorem}\label{OW}
Any two amenable unimodular random (decorated) graphs are decoration-equivalent.
%That is, if $\mu_E$ and $\mu_F$ are two amenable unimodular probability measures on $\Gstar$ then there is a unimodular random graph $(V,E),o;F)$ with $E,F\subset V^2$ such that $G_E=(V,E)$ and $G_F=(V,F)$ are connected, and $(G_E,o)$ has distribution $\mu_E$, $(G_F,o)$ has distribution $\mu_F$.
\end{theorem}
\noindent
One can view this theorem as a unimodular version of the Ornstein-Weiss theorem on the orbit equivalence of ergodic amenable free p.m.p. actions, \cite{OW}. 
%(two ergodic free p.m.p. actions of amenable groups are orbit equivalent)

%\begin{proof}
%By Corollary \def{unimod_spanning} and the fact that a unimodular spanning tree of a unimodular graph always has 1 or 2 ends, it is enough to prove the claim for unimodular spanning trees.
%Thus Lemma \ref{two_trees} finishes the proof.
%\end{proof}

%The next claim is a simple corollary of Lemma \ref{dual}. 

%\begin{corollary}\label{unimod_spanning}
%Leg $(G,o; T)$ be a unimodular random graph where $T$ is a spanning tree of $G$. Then $(T,o;G)$ is unimodular. In particular, $(G,o)$ and $(T,o)$ are decoration-equivalent.
%If $T$ is the Cayley diagram of $\Z$ (by one generator), then $(T,o; G)$ is an $\Aut(T)$-invariant decoration of $T$.
%\end{corollary}

%\begin{lemma}\label{two_trees}
%Let $T$ and $S$ be unimodular (decorated) random trees, and suppose that each has at most 2 ends. Then they are decoration equivalent. 
%\end{lemma}

\begin{proof} 
It is enough to prove that there exists a unimodular decoration of any amenable unimodular random graph $(G,o)$ with the standard Cayley diagram of $\Z$ (with some fixed vertex as root).
From the definition of amenability it follows that there is a sequence of coarser and coarser partitions ${\cal P}_n$ of $V(G)$ such that the sequence is jointly unimodular with $G$, every class of a ${\cal P}_n$ is finite, and any two vertices are in the same class of ${\cal P}_n$ for $n$ large enough. For $n=1,2,\ldots$, define an oriented path on every class of ${\cal P}_n$ uniformly at random, in such a way that the newly defined oriented paths contain those in ${\cal P}_{n-1}$. As $n\to\infty$ we obtain a copy of the oriented $\Z$ on $G$, just as wanted. 
\end{proof}

\begin{question}\label{uni_equivalence}
 What properties are preserved by decoration-equivalence? 
\end{question}
Besides amenability, treeability and cost
are trivially preserved. (At least this is the case if one defines $(G,o)$ to be treeable if it is decoration-equivalent to some unimodular random tree, and the cost of $(G,o)$ as $\inf \E(\deg_H (o))/2$ over all unimodular graphs $(H,o)$ that are decoration-equivalent to it.) In \cite{AHNR} it is shown that {\it strong soficity} (sofic approximability of every unimodular random marking of the graph) is also preserved by decoration-equivalence.
Furthermore, we are able to show that unimodular random trees of the same, finite expected degrees are decoration-equivalent. We do not present the proof here, because it is not in the scope of the present paper.
Similarly to decoration-equivalence of unimodular random graphs, one could look at the properties that are preserved by the following relation.
Say that two Cayley graphs (diagrams) $G$ and $H$ satisfy the relation $G\geq H$ if there is a random invariant decoration of $G$ with $H$, i.e., there is a random $F\subset V(G)^2$ that is $\Aut(G)$-invariant, and $(V(G),F)$ is almost surely isomorphic to $H$. %For Cayley diagrams, this relation is symmetric by Lemma \ref{dual}, and this special case was used in \cite{T2}. 
%How about the first $L^2$ Betti number?
%\begin{question}\label{invar_equivalence}
%Say that two Cayley graphs (diagrams) $G$ and $H$ satisfy the relation $G\geq H$ if there is a random invariant decoration of $G$ with $H$, i.e., there is a random $F\subset V(G)^2$ that is $\Aut(G)$-invariant, and $(V(G),F)$ is almost surely isomorphic to $H$. Is $\geq$ an equivalence relation? What properties are preserved by it?
%\end{question}
%By Gaboriau's theorem \cite{G}, no regular tree arises as an invariant decoration on some other regular tree of different degree.  

%\begin{question}
%Suppose that $G$ and $H$ are Cayley graphs, they have the same costs, and they are both treeable. Are they decoration equivalent?
%\end{question}
%A positive answer to Question \ref{treeeq} would imply a positive answer to the previous question. In particular, consider the trees that produce the cost for $G$ and for $H$, compose the decorations and apply the independent coupling.

\def\Fcal{{\cal F}}

\section{Dyadic fiid partitions with some connected pieces}\label{s.diadic}

If $\calp$ if some partition of a set $X$, denote by $\calp (x)$ the class of $x\in X$ in $\calp$.

\begin{lemma}\label{Tpartition}{\bf (Partitioning a 1-ended tree)}
Let $T=(T,o)$ be an ergodic unimodular random tree with one end and degrees at most $d$. Then there exists a fiid nonempty subset $U\subset V(T)$ that induces a connected subgraph of $T$, and a fiid sequence $(\calq_n)$ of coarser and coarser partitions of $U$ such that any $x,y\in U$ is in the same class 
of $\calq_n$ if $n$ is large enough, and 
for every $x\in U$
there are infinitely many $n$ such that $T|_{\calq_n (x)}$ is connected and has $2^n$ elements.
\end{lemma}
Let us emphasize that the real difficulty in the above lemma is the requirement that the restriction of $T$ to many of the dyadic parts be connected. Without the connectedness requirement, the lemma would easily follow from Section 4 of \cite{T}.

\begin{proof}
Fix an arbitrary sequence $n_1<n_2<n_3<\ldots$ with the property that $n_{i+1}/n_i>3+2\log (d)$ for every $i$. 
From now on, $\tau$ will always denote an arbitrary subtree of $T$.
Given some $x\in V(\tau)$, define $\tau_x$ as the finite subtree induced by vertices in $\tau$ that are separated from infinity by $x$ (including $x$). 
Observe that for every $x\in V(\tau)$
\begin{equation}
|\tau_x|\leq 1+(d-1)\max_{x\sim y,y\in \tau_x} |\tau_y|,
\label{eq:Txbound}
\end{equation}
where $x\sim y$ indicates that $x$ and $y$ are adjacent in $\tau$.
For $k>0$, define $S_k(\tau)=S_k=\{x\in V(\tau):\, |\tau_x|\geq 2^k\}$. The set $S_k$ induces a subtree of $\tau$; we will refer to this subtree as $\tau|_{S_k}$. Let $L_i(\tau)$ be the set of vertices of $\tau|_{S_{n_i}}$ that have degree 1 in $\tau|_{S_{n_i}}$.
If $x\in L_i(\tau)$, then we have
\begin{equation}
2^{n_i}\leq |\tau_x|\leq 1+(d-1)2^{n_i},
\label{eq:Libound}
\end{equation}
using \eqref{eq:Txbound} and the fact that the neighbors of $x\in L_i(\tau)$ in $\tau_x$ are not in $S_{n_i}$.
Note also that if $x,y\in L_i (\tau)$, $x\not=y$, then $\tau_x\cap \tau_y=\emptyset$.

For sets $A$ and $B$, say that $A$ {\it cuts} $B$ if $A\cap B\not=\emptyset$ and $B\setminus A\not=\emptyset$. If $\Pcal$ is a partition, say that $\Pcal$ cuts $B$ if some class $A\in\Pcal$ cuts $B$.

Define $\Pcal_0^{0}$ as the partition consisting of classes $\{x\}$, $x\in V(T)$. Given $i\in \N$, suppose that partition $\Pcal_{j}^{i-1} $ has been defined for every $0\leq j<i$, and that 
\begin{enumerate}[label={(\arabic*)}]
\item every class of $\Pcal_{j}^{i-1}$ induces a connected subgraph of $T$;
%\item every class of $\Pcal_{j}^{i-1}$ is either a singleton, or has $2^{j}$ elements;
\item whenever $j'<j$, $\Pcal^{i-1}_j$ is coarser than $\Pcal^{i-1}_{j'}$. In other words, $\Pcal_{j}^{i-1}$ does not cut any class of $\Pcal_{j'}^{i-1}$. 
\item Every class of $\Pcal^{i-1}_{j}$ either has $2^{n_{j}}$ elements, or it is a singleton.
\end{enumerate}
We will define $\Pcal^{i}_j$ for every $j\leq i$, such that $(1)-(3)$ is satisfied by $\{\Pcal^{i}_j, \, j\leq i\}$ (with $i-1$ replaced by $i$). 
Let $\tau$ be an arbitrary subtree of $T$. Since (1) and (2) are satisfied, it is easy to check that the restriction of 
$\Pcal_{j}^{i-1}$ to $\tau$ also satisfies (1) and (2).

For every $x\in L_i (\tau)$, we will take a connected subgraph $C_x (\tau)=C_x$ of $\tau_x$ such that $x\in C_x$, and $|C_x|=2^{n_i}$. Furthermore, $C_x$ will be such that for every $j<i$ the number of classes in $\Pcal_{j}^{i-1}$ that are cut by $C_x$ is at most 1. Such a $C_x$ exists for the following reason. Starting from $c_1:=\{x\}$, define a sequence of {\it connected} subgraphs $c_1\subset c_2\subset c_3\subset\ldots$ in $\tau_x$, by always adding one vertex to the previously chosen set (so $|c_k\setminus c_{k-1}|=1$). The rule for this sequence is that whenever our current set $c_k$ cuts some class $\pi$ of $\Pcal_{j}^{i-1}$ for some $j<i$, we will add vertices from $\pi$ to the sequence as long as $\pi$ gets fully contained in some $c_m$. There may be more than one $j$ and $\pi (j)\in\Pcal_{j}^{i-1}$ that is cut by $c_k$, but by $(2)$ there is always a smallest one with regard to containment, and so we only need to follow the above rule for this smallest $\pi$ until it gets fully contained.
Such a growing sequence exists, and by definition, it always cuts at most one class of $\Pcal_{j}^{i-1}$ for every $j<i$. 
We have just seen that a $C_x$ as above exists; choose $C_x$ randomly (using some pre-fixed fiid rule) among all the possible choices satisfying the above constraints. The $C_x$ are pairwise disjoint over $x\in L_i (\tau)$, because the $\tau_x$ are pairwise disjoint.

We introduce notation $T^{(1)}:=T$. Choose $C_x=C_x(T)$ as defined above, for every $x\in L_i (\tau)$.
Consider the infinite component $T^{(2)}(i)=T^{(2)}$ of $T\setminus L_i(T)$, and repeat the above procedure for this new tree $T^{(2)}$, using the restriction of $\Pcal_{j}^{i-1}$ to it. 
%(We hide dependence of the $i$ for the sake of simpler notation.) 
%We can define $T_x^{(2)}$, $S_{n_i}(T^{(2)})$, $L_i(T^{(2)})$ as above, and apply the above procedure, using the restrictions of the $\Pcal_{j}^{i-1}$ to $T^{(2)}$, 
These restrictions still satisfy (1) and (2), and we can define $C_x(T^{(2)})$ for every $x\in L_i (T^{(2)})$. (Note that condition (3) was not used in the above construction of $C_x$.)

Continue and repeat the above procedure for the series $T^{(k)}(i)=T^{(k)}$ as $k=3,4,\ldots$, defined recursively and similarly to the way we defined $T^{(2)}$. For every $x\in V(T)$ there is at most one step $k$ in which 
$x\in L_i (T^{(k)})$, and hence
$C_x=:C_x^{(k)}$ is defined.
One arrives to a family of pairwise disjoint connected subgraphs $C^{(k)}_x$ of $T^{(k)}$ as $x\in L_i(T^{(k)})$, of size $2^{n_i}$ each. The family of all these $C^{(k)}_x$ (as $k=1,2,\ldots$, and $x\in L_i(T^{(k)})$), together with all the singletons not contained in any of them, defines a partition of $T$. Call this partition $\Pcal_i^{i}$. (We will define the $\Pcal_j^{i}$ with $j<i$ later.)
We mention that 
\begin{equation}
\P(o\in \cup_k L_i (T^{(k)}))\leq 2^{-n_i},
\label{eq:uj}
\end{equation}
as can be seen by an easy mass transport argument where $x\in L_i(T^{(k)})$ sends mass 1 to every point of $T_x^{(k)}$.

Now, 
\begin{equation}
\P(o \text{ is in some } C^{(k)}_x)\geq \inf_{k,x} |C^{(k)}_x|/|T^{(k)}_x|\geq 2^{n_i}/(d2^{n_i})\geq
1/d
\end{equation}
using 
\eqref{eq:Libound} and a standard MTP argument.
So we have that 
\begin{equation}
\P(o\text{ is in a class of size } 2^{n_i}\text{ of } \Pcal^{i}_i)\geq 1/d
\label{eq:first}
\end{equation}

\def\supess{{\rm supess}}

For every $j<i$, define $\Pcal^{i}_j$ to be the collection of all classes in $\Pcal^{i-1}_j$ that are not cut by any of the classes in $\Pcal^{i}_i$, and let all vertices outside of these classes be singletons in $\Pcal^{i}_j$. 
The $\Pcal^{i}_j$ defined this way satisfy (1), (2) and (3) with $i-1$ replaced by $i$, whenever $j\leq i$. Furthermore, $\Pcal^{i-1}_j$ is always coarser than $\Pcal^{i}_j$ by definition, and all its classes are finite, hence the sequence $\Pcal^{i}_j$ has a weak limit as $i\to\infty$.

So define the weak limit $\Q_j:=\lim_{i\to\infty}\Pcal^{i}_j$. The partition $\Q_j$ inherits the property that every class of it induces a connected subgraph of $T$.
Then, using \eqref{eq:first}, we have
$$
\P(o\text{ is in a class of size } 2^{n_j}\text{ of } \Q_j)\geq \P (|\Pcal_j^{j}(o)|=2^{n_j})-\sum_{i>j} \P(\Pcal_j^{j}(o) \text{ is cut by } \Pcal_i^{i})
$$
\begin{equation}
\geq
1/d-2\sum_{i=j+1}^\infty d 2^{n_j}/2^{n_i},
\label{eq:second}
\end{equation}
as we explain next.
There are {\it two ways} for the class $\pi=\Pcal_j^{j}(o)$ to be cut by $\Pcal_i^{i}$. The first one is if $L_i(T^{(k)})\cap\pi\not=\emptyset$ for some $k\in\N$. If this is not the case, there is another way: if there is an $x\in V(T)$ such that
$\pi$ is the (single) class of $\Pcal_j^{j}$ that $C_x\in \Pcal_i^{i}$ (as in the above construction) cuts. The probability for the ``first way'' can be bounded using Chebyshev's inequality:
$$\P(L_i\cap\pi\not=\emptyset)\leq \E(|L_i\cap\pi|),
$$
where $L_i:=\cup_k L_i(T^{(k)})$.
For the right hand side we have the following upper bound, where $\supess |\pi|$ denotes the essential supremum of the size of $\pi=\Pcal^j_j(o)$ with regard to the random graph and partition (in our case we have $\supess |\pi|\leq d2^{n_j}$).
$$\E(|L_i\cap\pi|)\leq \sum_\ell \P(|\pi|=\ell)\sum_{m=0}^\ell \P(|\pi\cap L_i|=m\bigl| |\pi|=\ell)m\leq$$
$$\leq \supess |\pi|\sum_\ell \P(|\pi|=\ell)\sum_{m=0}^\ell \P(|\pi\cap L_i|=m\bigl| |\pi|=\ell)m/\ell\leq $$
$$\leq d2^{n_j} |\pi|\sum_\ell \P(|\pi|=\ell) \P(o\in L_i\bigl| |\pi|=\ell)\leq
d2^{n_j} |\pi|\P(o\in L_i)\leq d2^{n_j}2^{-n_i},
$$
using \eqref{eq:uj} for the last inequality. To summarize, 
\begin{equation}
\P(L_i\cap\pi\not=\emptyset)\leq d2^{n_j}2^{-n_i}.
\label{eq:firstway}
\end{equation}
Now let us continue with an upper bound on the probability of the ``second way'' for $\pi=\Pcal_j^{j}(o)$ to be cut by $\Pcal_i^{i}$, assuming that it is not cut by the set $L_i$. 
Note that $\{T^{(k)}_x:x\in L_i(T^{(k)}), k\in\N^+
\}$ is a partition of $T$. 
So the assumption that $\pi$ is not cut by $L_i$ implies that there is a unique $k$ and $x\in L_i(T^{(k)})$ such that $\pi\subset T^{(k)}_x$. We know that $C_x (k)$ cuts at most one of the classes of $\Pcal_j^{j}$. If there exists such a class, denote it by $\bar\pi$, otherwise let $\bar\pi:=\emptyset$. Then, using the MTP for the first inequality,
$$\P(\pi \text{ is cut by }C_x (k) |L_i\cap\pi=\emptyset)=\P(o\in \bar\pi|L_i\cap\pi=\emptyset)\leq \frac{|\bar\pi|}{|T_x|}\leq \frac{2^{n_j}}{2^{n_i}}.
$$
To see \eqref{eq:second}, note that the sum of this and the right side of \eqref{eq:firstway}, over all $i>j$,
gives an upper bound for the total probability of $\Pcal_j^{j}(o)$ being cut in a later stage.

\def\Scal{{\cal S}}

The inequality \eqref{eq:second} shows that for every $j$, $o$ is in a non-singleton class of $\Q_j$ with probability at least $1/2d$ (using the condition on the $(n_i)$). By the Borel-Cantelli lemma, with positive probability $o$ is contained in a non-singleton class of $\Q_j$ (that is, a class of size $2^{n_j}$) for infinitely many $j$, and we know that all these classes induce a connected graph in $T$. By the ergodicity of $T$, there exist vertices with this property almost surely; let their set be $U$. 
On the event $o\in U$, the same holds for every element $x\in \Q_j (o)$ for every such $j$. From this 
we can conclude that $T|_U$ is connected. Finally we show that $U$ is a fiid. The sequence $(\Q_j)$ was defined as a fiid. Note that
$U_n:=\bigl\{x\in V(T): \bigl|\{i:|\Q_i (x)|=2^{n_i}\}\bigr|\geq n
\bigr\}$ is a fiid, $U_{n+1}\subset U_n$, and $U=\lim U_n$. Hence $U$ is also a fiid.
\end{proof}

Although we are not going to use the next lemma, we claim it as a straightforward consequence of Lemma \ref{Tpartition}.
\begin{lemma}\label{Gpartition}{\bf (Partitioning an amenable graph)}
Let $G=(G,o)$ be an ergodic amenable unimodular random graph with one end and degrees at most $d$. 
Then there exists a fiid nonempty subset $U\subset V(G)$ that induces a connected subgraph of $G$, and a fiid sequence $(\calp_n)$ of coarser and coarser partitions of $U$ such that any $x,y\in U$ is in the same class 
of $\calp_n$ if $n$ is large enough, and 
for every $x\in U$
there are infinitely many $n$ such that $T|_{\calp_n (x)}$ is connected and has $2^n$ elements.
\end{lemma}

\begin{proof}
Choose an fiid 1-ended spanning tree of $G$, as in \cite{T3}.  Apply Lemma \ref{Tpartition} to this spanning tree.
\end{proof}

\section{Representation of a unimodular one-ended tree by a tiling; generalization to amenable graphs}\label{s.treetiling}

\def\Vdiad{V^{{\rm top}}}
\def\Vconn{V^{{\rm top}}}
\def\Tcal{{\cal T}}
\def\Tcaleq{{\cal T}_{\rm eq}}
\def\Rcal{{\cal R}}
\def\f{f}
\def\Cube{{\rm Cube}}

Given a copy of the graph or diagram $\Z^d$, say that we {\it expand it to} $\R^d$ if we fill up every cube (as a subgraph) of $\Z^d$ with a unit cube of $\R^d$ (as a polyhedron). In other word, we think about $\Z^d$ as being embedded in $\R^d$ in the usual way. If $x$ is a vertex of the $\Z^d$ that was expanded to $\R^d$, we assign a cube $\Cube _x=x+[-1/2,1/2]^d$ to $x$.

\begin{theorem}\label{Tdomains}{\bf (Fiid tiling representation of a one-ended tree)}
Let $T$ be an ergodic unimodular random tree with one end and degrees at most $D$. Then there is an fiid decoration of a fiid subset $U$ of $V(T)$ with the Cayley diagram $\Z^d$, (meaning that a cubic grid is defined on $U$ as vertex set). Furthermore, if $d\geq 2$, there is a tiling of $\R^d$ that represents $T$, and a bijection from $V(T)$ to the tiles, and this bijection as well as its inverse preserves adjacency. The tiling and the bijection are both given as a fiid from $T$.
%Looking at this copy of the points of $\Z^3$ as a subset of $\R^3$, one can further assign tiles in $\R^3$ to the points of $V(T)$ in such a way that the tiles give a tiling of $\R^3$ and it represents $T$. This assignment can also be defined as a fiid map from $V(T)$.
\end{theorem}

We remark that the copy of $\Z^d$ and the tiling in the theorem is defined only up to isometries. %Also note that the theorem works for $d=2$, unlike most of our results.

\begin{proof}
For the ease of notation we only prove the theorem for $d=3$, but the same argument applies to any $d\not= 3$.
The claim about the decoration by $\Z^3$ follows from Theorem 4.1 in \cite{T}. We provide a proof here because the notation will be needed for the second part of the theorem.
As before, for $x\in V(T)$, let $T_x$ be the (finite) subgraph of $T$ induced by those vertices of $T$ that are separated from infinity by $x$ (including $x$).
For $x,y\in V(T)$, write $y\leq x$ if $y\in T_x$, and write $y< x$ if $y\leq x$ and $y\not= x$.
Let $\Q_n$ and $U$ be as in Lemma \ref{Tpartition}; recall that $T|_U$ is connected.
Let $\Vdiad$ be the set of $x\in U$ such that for some $m$, $T|_{\Q_m (x)}$ is connected and $y\leq x$ for every $y\in {\Q_m (x)}$.
%In particular, elements of $\Vdiad$ are maximal in some $\Q_m(x)$ with respect to the partial order ``$<$" on $V(T)$.
%and $T|_{\Q_m (x)}$ is connected. 
For every $x\in\Vdiad$ choose a maximal $m=m(x)$ with this property.
It is straightforward from Lemma \ref{Tpartition} that $\Vdiad$ is nonempty. 
Let $\Vdiad_1=\bigl\{x\in\Vdiad :\,
T_x\cap \Vdiad=\{x\}\bigr\}$.
Define $\Vdiad_n$ recursively, to be the set of points $x\in\Vdiad$ such that there exists a $y<x$ with $y\in\Vdiad_{n-1}$ and such that for every $y<x$ with $y\in \Vdiad$, we have $y\in\Vdiad_{k}$ for some $k<n$. If $x\in \Vdiad_n$, then for every $y\in\Vdiad_n\cap \Q_m (x)$ we have $m(y)<m(x)$, hence $\Q_{m(y)}\subset Q_{m(x)}$. An example is shown on the left side of Figure \ref{farepi}.

From now on, $\Z^3$ will stand for the Cayley diagram.
First we define a fiid copy of the graph $\Z^3$ on the vertex set $U$ (that is, a fiid decoration of $U$ in $G$ by $\Z^3$). As $n=1,2,\ldots$, for every element $x$ of $\Vdiad_n$, define a $2^{\lfloor m(x)/3\rfloor}\times 2^{\lfloor (m(x)+1)/3\rfloor}\times 2^{\lfloor (m(x)+2)/3\rfloor}$ grid $L(x)$ on $\Q_{m(x)}(x)=\Q_{m}(x)$ as vertex set, with edges oriented and colored by the first, second and third generator of $\Z^3$, respectively, as in the Cayley diagram of $\Z^3$, and in such a way, that it respects all the $L(y)$ defined in earlier steps (in other words, if $y<x$, $y\in \Vdiad$, then 
define $L(x)$ so that $L(y)\subset L(x)$ and the colors and orientations agree). This is possible because of the previous paragraph, and since the $L(y)$ also have a dyadic form but smaller size.
Beyond this constaint, the adjacencies of the actual points of $\Q_m(x)$ in this grid can be determined arbitrarily, but one should follow some fixed fiid rule (as usual). By Lemma \ref{Tpartition}, every point $u\in U$ is in infinitely many $\Q_m(x)$. Hence the limit of the $L(x)$ has to be $\Z^3$, or an infinite half-space, quarter-space or eigth-space of $\Z^3$ -- however, only the first one is possible, by a simple MTP argument (otherwise one could assign points of the border to the vertices, in a way that the same border point is assigned to infinitely many points). Again by Lemma \ref{Tpartition}, any two points of $U$ are in the same $\Q_m$-class if $m$ is large enough. Therefore, the limit of the $L(x)$ is in fact a connected copy of $\Z^3$ on $U$, which we defined as a fiid. 
%Define the map $\lambda:U\to\Z^3$ using this copy, through setting $\lambda (o)=0$.

In the rest of the proof we explain how to decorate every point
$v$ of $V(T)$ by a tile $\tau (v)$ of $\R^3$, with $\R^3$ given as the expansion of the $\Z^3$ that we just constructed on $V(T)$. These tiles will be polyhedra, and moreover, {\it bricks} (higher dimesional rectangles) with finitely many possible holes in them, such that the holes are also bricks. The collection $\{\tau (v):\, v\in V(T)\}$ will give a partition of $\R^3$, with adjacency relation isomorphic to that of $T$. In other words, we will define a representation of $T$ by a tiling. This will be done as a fiid.
%The representation by a tiling will be defined using $\lambda$ (thinking about $\Z^3$ as a graph naturally embedded into $\R^3$). Hence one will be able to think of it as a decoration of $\lambda(v)\in\Z^3$ with tile $\tau (v)$ (plus $v$). DUALITAS LEMMA ITT JON BE 

\def\thin{{\rm Thin}}
\def\F{{\cal F}}

\begin{figure}[htbp]
\vspace{0.2in}
\begin{center}
\includegraphics[keepaspectratio,scale=0.60]{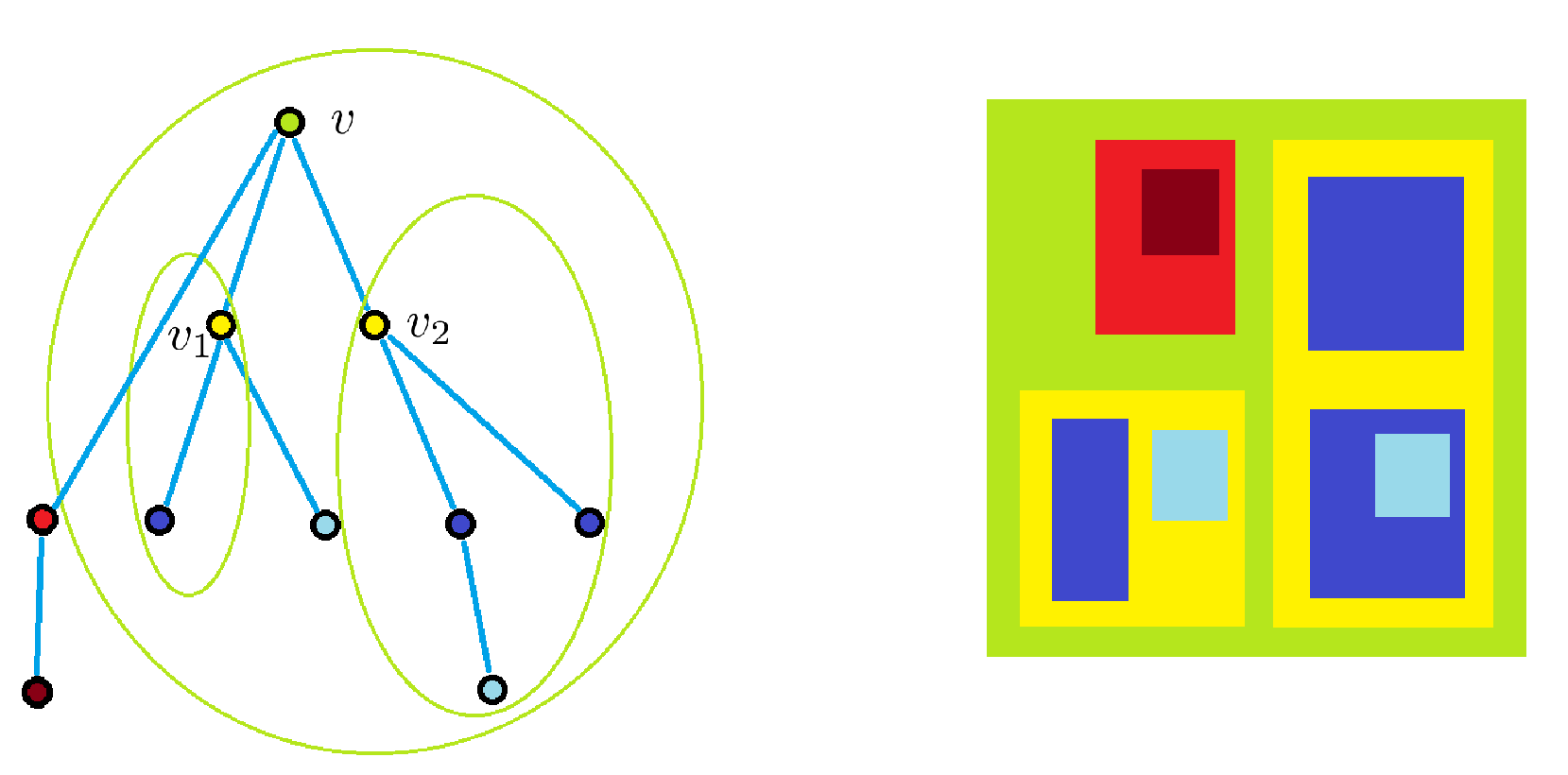}
\caption{Representing $T$ by a tiling. Here $v_1, v_2\in\Vdiad _1$, and $v\in \Vdiad_2$, $n_i=i$ is assumed. $\Q_{m(v_1)}(v_1)=\Q_1(v_1)$, $\Q_{m(v_2)}(v_2)=\Q_2(v_2)$ and $\Q_{m(v)}(v)=\Q_3(v)$ are circled.
When we start the construction of the representation of $T_v$, the representation of $T_{v_1}$ and $T_{v_2}$ is already given (the yellow domains and the ones surrounded by them on the right).
}\label{farepi}
\end{center}
\end{figure}

For a polyhedron $H$ and $\eps>0$, let $\thin (H,\eps)$ be the subset of $H$ of points at distance at least $\eps$ from the complement of $H$. Recall $\Vdiad\subset U\subset V(T)$; now we will partition $V(T)$ to pieces so that every piece corresponds bijectively to one point of $\Vdiad$. Namely, as $x\in \Vdiad$, let $\F (x)=\{y:\, y\in T_x,$ there is no element of $ \Vdiad$ but $x$ on the $x-y\text{ path in } T
\}$.
The set $\{\F (x):\, x\in \Vconn\}$ is a partition of $V(T)$. Before we proceed, let us highlight a certain property of $\Vconn$. 
The set $\Vdiad$ inherits a tree structure $\Tcal (\Vdiad)$ from $T$: let $v\in U$ and $w\in U$ be adjacent if $w$ separates $v$ from infinity in $T$, and any other vertex of $\Vdiad$ either separates both $v$ and $w$ from infinity, or none of them.
Consider the subgrid $L(x)$ (as defined above). This contains all the $L(y)$ for $y\in T_x\cap \Vdiad$. 
%We will expand these subgrids first, to obtain a brick, and then we remove all the smaller bricks inside it, that correspond to the elements of $\Vdiad$ that $x$ separates from infinity in $\Tcal (\Vdiad)$. This way we will get a tiling in the limit, that represents $\Tcal (\Vdiad)$. We turn this into a tiling representation of $T$ next.

As $k=1,2,\ldots$, for every $x\in \Vdiad _k$ we will define the tile $\tau (x)$.
First, define recursively 
$$\tau'(x)=\thin(\cup_{v\in L(x)} \Cube _v , 2^{-k})\setminus \cup_{y<x, y\in\Vdiad} \tau' (y).$$
Now extend the definition to all $z\in V(T)$, by defining it for every $z\in\F (x), z\not=x$, $x\in \Vconn$ as follows. Fix $x\in\Vdiad$. Define a polyhedral subset $\tau' (z)$ of $\tau' (x)$, such that it satisfies the following properties, but otherwise arbitrarily. For every $y\in\Vdiad$ with $y<z<x$, we will have $\tau'(y)\subset \tau' (z)\subset \tau'(x)$, and furthermore, $\partial\tau'(y)\cap\partial\tau'(z)=\emptyset$ and 
$\partial\tau'(z)\cap\partial\tau'(x)=\emptyset$. Finally, if $z'\in \F (x)$ and
$z'<z$ then $\tau'(z')\subset \tau'(z)$ and $\partial\tau'(z')\cap\partial\tau'(z)=\emptyset$.  Such a definition of the $\tau' (z)$ is possible locally, and we can make it an fiid by fixing some central rules for these choices. Finally, for each $v\in V(T)$, define $\tau (v):=\tau'(v)\setminus\cup_{w\in T_v} \tau'(w)$. See Figure \ref{farepi} for a summary of the construction.

The properties in Definition \ref{maindef} are trivially satisfied by the above construction (with tiles that are bounded polyhedra with finitely many 0-faces), only the third one requires some reasoning. So consider some fixed ball $B$ in the $\R^3$ that arose by expanding $\Z^3$. 
We may assume that the center of $B$ is in $\Z^3$. 
Choose $x\in U$ such that $V(T_x)$, viewed as a subset of $\Z^3\subset\R^3$, has convex hull that contains $B$. By construction, then the only tiles that may intersect $B$ 
are $\tau (y)$, $y\in T_x$. 
%To summarize, the $\{\tau(v):v\in T(V)\}$ will be a set of bricks in $\R^3$ surrounding each other in a tree-like fashion. Defining adjacency between $\tau '(v)$ and $\tau (v')$ if $\tau '(v)\subset\tau' (v')$ and for no $w\in V(T)\setminus\{v,v'\}$ is $\tau '(v)\subset \tau'(w)\subset\tau' (v')$, we get a directed graph on $\{\tau'(v):v\in T(V)\}$ that is isomorphic $T$ with all edge directed towards the single end. 
\end{proof}

\def\Vor{{\rm Vor}}

The next lemma will be needed to extend Theorem \ref{Tdomains} and construct an invariant tiling representation of an arbitrary amenable unimodular graph with one end.

Let $\eps>0$ be arbitrary, $A,B\subset \Rr$. 
Define $v (A,B,\eps):=\{x\in\Rr\,:\, \dist (x,A)<\min (\dist (x,B),\eps)
\}$, and let $\Vor (A,B,\eps)$ be the connected component of $v(A,B,\eps)$ containing $A$. 
Say that a collection $P$ of polyhedra represents a finite graph $G$ if there is a bijection mapping to every vertex $x\in V(G)$ a polyhedron $P(x)\in P$ such that $P(x)$ and $P(y)$ share a hyperface if and only if $x$ and $y$ are adjacent.

\begin{lemma}\label{addedge}{\bf (Adding an edge)}
Let $G_0$ be a finite connected graph represented by a collection $P$ of polyhedra, and suppose that every polyhedron in $P$ has finitely many 0-faces. Let
$\Delta\subset\R^d$ be the union of the closures of these polyhedra. 
Let $x,y\in V(G_0)$, and $I$ be a path in $G_0$ between $x$ and $y$. Then the new graph $G_0\cup\{\{x,y\}\}$ can also be represented in $\Delta$, by a collection $P'$ of polyhedra such that $P(v)=P'(v)$ whenever $v\not\in I$. Furthermore, let $\gamma_{xy}$ be a broken line segment in $\R^d$ connecting $P(x)$ to $P(y)$, and suppose that there exists an $\eps_0>0$ such that the $\eps_0$-neighborhood $N(\gamma_{xy},\eps_0)$
of $\gamma_{xy}$ is contained in $\Delta$.
Then $P'$ can be constructed so that if we restrict $P$ and $P'$ to $\Delta\setminus N(\gamma_{xy},\eps_0)$ then they coincide (in other words, 
outside of $N(\gamma_{xy},\eps_0)$, the polyhedra of $P$ are unchanged when $P'$ is constructed). Finally, all of the above can be constructed as a function of $P$ so that it is equivariant with any isometry of $\R^d$.
\end{lemma}

\begin{proof}
It is enough to prove the more restrictive version of the claim, when $\gamma_{xy}$ is given. 
Figure \ref{kocsany} gives an intuitive summary of the coming proof: one can grow a path between $P(x)$ and $P(y)$ while only modifying the $P(v)$ with $v\in I$. 

To explain formally what is summarized on Figure \ref{kocsany}, consider $\gamma=\gamma_{xy}$.
For $A\subset \R^d$, $c>0$, denote by $N^\infty(A,c)$ the $c$-neighborhood of $A$ in the $L^\infty$ norm.
To define this norm, we have to fix the coordinate axes as a deterministic function of P, which is clearly possible by some properly chosen rule that uses the finitely many extremal points of the polyhedra of P.
So in the remainder of this proof, all distances are understood as $L^\infty$, with axes fixed as a fiid from $G$. (The point of using this distance is to have the neighborhoods of line segments be polyhedra.)
Let $ D_1,\ldots, D_m$ be the tiles that $\gamma$ crosses, in this consecutive order as we go from $x$ to $y$.
Pick an $\eps>0$ such that 
$N^\infty(\gamma, \eps)\subset
N(\gamma_{xy},\eps_0)$, and
$D_i\setminus N(\gamma,\eps)$ is connected for every $i$.

Define $\Gamma_1=N^\infty(\gamma, \eps/2)$ and
$D_1'=P(x)\cup \Gamma_1\setminus P(y)$. 
Suppose $\Gamma_{j-1}$ and $D_{j-1}'$ have been defined. Let 
%$\Gamma_j=(\cup_{i=j}^m D_i)\cap \Vor (\Gamma_{j-1}, P(y), (1-2^{-j})\eps)$ 
$\Gamma_j=\Vor (\cup_{i=1}^{j-1}\Gamma_i, P(y),2^{-j}\eps)$
and
$D_j'=D_j\cup\Gamma_j\setminus (\bigcup_{i<j}D_i\cup \Gamma_{j-1})$, where we use the $L^\infty$ distance in the definition of $\Vor$.
To finish, for every $v\in V(G)$ such that $P(v)=D_j$ for some $j$, define $P'(v)=D_j'$. For all other $v\in V(G)$, let $P'(v)=P(v)$.
\end{proof}

\begin{figure}[h]
%\vspace{0.2in}
\begin{center}
\includegraphics[keepaspectratio,scale=0.65]{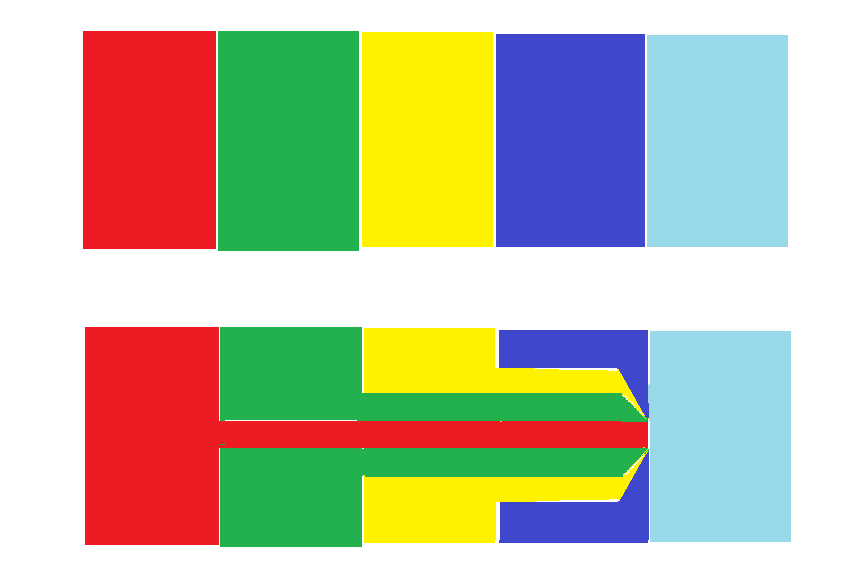}
\caption{Adding an extra edge to a representation by tiles. Note that in dimension at least 3, the slim tunnel between the two tiles does not disconnect the other tiles.}\label{kocsany}
\end{center}
\end{figure}

\def\M{{\cal M}}

The next theorem contains Theorem \ref{mackovilag} as a special case. So far we have only defined tiling representation for transitive graphs or diagrams (Definition \ref{maindef}). If $(G,o)$ is a unimodular random graph, and if a tiling $\calP$ of $\R^d$ is given, one would have to choose a root for $\calP$ to be able to say that $\calP$ is a tiling that represents $(G,o)$. Suppose that $\calP$ is an $\Isom (\R^d)$-invariant tiling of $\R^d$, and suppose that the mean $m$ of ${\rm Vol}(\calP (0))^{-1}$ is finite, where $\calP (0)$ is the tile containing the origin and ${\rm Vol}$ is the volume function.
Choose a random point $t_P$ uniformly in every tile $P\in\calP$. The random set $\{t_P:P\in\calP\}$ is $\Isom (\R^d)$-invariant and it has finite intensity (its intensity is actually $m$, as shown by the continuous version of the MTP). 
Call this random point set the {\it point process generated by} $\calP$.

\begin{definition}{\bf (Tiling representation, unimodular random graphs)}
Let $(G,o)$ be a unimodular random graph, $\calP$ be an $\Isom (\R^d)$-invariant random tiling of $\R^d$, and suppose that $\E({\rm Vol}(\calP (0))^{-1})$ is finite. Consider the point process generated by $\calP$ and condition on having a point in 0 (that is, take the Palm version). If the adjacency graph of $\calP$ with root in $\calP (0)$ is distributed as $(G,o)$, then we say that the tiling $\calP$ {\it represents} $(G,o)$.
\end{definition}

\begin{theorem}\label{Gdomains}{\bf (Amenable unimodular random graphs as tilings)}
Let $G$ be an ergodic amenable unimodular random graph that has one end almost surely, and degrees at most $D$. Then for $d\geq 3$ there is an $\Isom (\R^d)$-invariant random tiling of $\R^d$ that represents $G$, such that every tile is bounded and has finitely many 0-faces. 
The representation can be viewed as a factor of iid map from $V(G)$ to the set of tiles.
\end{theorem}

\begin{proof}
We prove the second assertion first.
Let $T$ be an fiid 1-ended spanning tree percolation of $G$ (that is, $T\subset G$ and $(G,T)$ is unimodular). The existance of such a $T$ is shown in \cite{T3}. (The weaker claim that there is an invariant random spanning tree with 1 or 2 ends was shown in \cite{AL}.) Apply Theorem \ref{Tdomains} to obtain a fiid decoration of $T$ by a tiling $\M$  that represents $T$ by a tiling
in $\R^d$. For every edge $e\in E(G)$ define 
$\pi_e$ to be the path in $T$ that connects the endpoints of $e$.

We will define a collection of broken line segments (``curves'') $\gamma_e\subset \R^d$ as $e\in E(G)\setminus E(T)$, with the following properties:
\begin{enumerate}[label={(\arabic*)}]
\item No two of these curves will intersect. 
\item If $v_0,\ldots, v_n$ are the vertices of $\pi_e$ in consecutive order (so in particular, $v_0$ and $v_n$ are the endpoints of $e$), then $\gamma_e$ connects $\M(v_0)$ and $\M(v_n)$, and crosses no other polyhedron in $\M$ but the $\M(v_i)$, in consecutive order. 
\end{enumerate}
Once this family $\{\gamma_e:\,e\in E(G)\setminus E(T)\}$ is constructed, two consequences will follow. First, every $\M(x)$ is crossed by only finitely many of the $\gamma_e$ (because $x$ is in finitely many of the $\pi_e$). This implies the second important property: the infimum of the distance of a $\gamma_e$ from all the other $\gamma_{e'}$ is positive (because only finitely many of the $\gamma_{e'}$ enter the finitely many polyhedra that $\gamma_e$ crosses). Let this infimum be $\eps_e$($>0$). Now, apply Lemma \ref{addedge} to a $\gamma_e$, with $\eps:=\eps_e/3$, and $P$ the set of tiles in $\M$ that $\gamma_e$ intersects.
Then the modification of the polyhedra of $\M$ happens in pairwise disjoint parts of $\R^d$, hence they can be done simultaneously for all the $e$. Also, every tile of $\M$ is being modified finitely many times (for exactly those $e$ where $\gamma_e$ intersects the tile).
The original $\M$ had the property that any ball in $\R^d$ intersects only finitely many tiles of it. Hence the modified $\M$ will also have this property, because each of the tiles of $\M$ is contained in the union of finitely many new tiles. It follows that the resulting tiling is an fiid representation of $G$, as in Definition \ref{maindef}.
Thus, to obtain a fiid decoration of $G$ by a tiling that represents it, we only have to define the family of curves that satisfies properties (1) and (2) above.

%Let $\bigl((a_n,b_n)\bigr)_{n=1}^\infty$ be some arbitrary enumeration of all the pairs of positive integers.
Let $E_n$ be the set of edges $e$ in $E(G)\setminus E(T)$ that have the following properties: the total number of points in the finite components of $T\setminus \pi_e$ is less than $n$ and $|\pi_e|\leq n$. One can check that
\begin{itemize}
\item the graph $\cup_{e\in E_n} \pi_e$ has only finite components;
\item $\cup_n E_n=E(G)\setminus E(T).$
\end{itemize}
The second item is trivial. To see the first item, suppose that $\cup_{e\in E_n} \pi_e\subset T$ has an infinite component C. 
Then any infinite path in $T$ has all but finitely many points in $C$.
In particular, there is a subpath $x_1,\ldots,x_{2n}$ in $C$ such that $x_1$ separates each of
$x_2,\ldots,x_{2n}$ from infinity. Let $e\in E_n$ be such that $x_1\in \pi_e$. Then the finite 
components of $T\setminus\pi_e$ contain at least $2n-|\pi_e|\geq n$ elements 
of $\{x_1,\ldots, x_{2n}\}$, contradicting the definition of $E_n$.

%The first item follows from noticing that any path $\pi_e$ has a unique point $x$ that is ``closest to infinity'' in $T$, meaning that no point of $\pi_e$ is in the infinite component of $T\setminus\{x\}$. Hence for any two paths $\pi_e$ and $\pi_{e'}$ that intersect each other, one of them is such that its point closest to infinity also separates the other path from infinity in $T$. From this it is easy to deduce the first item above.

As $n=1,2,\ldots$, do the following. For every (finite) component $C$ of $\cup_{e\in E_n} \pi_e$ and $e$ such that $\pi_e\cap C\not=\emptyset$,
define the path $\gamma_e$ such that it satisfies property (2) above. Furthermore, do this in such a way that $\gamma_e$ does not intersect any of the (finitely many) other $\gamma_{e'}$ with $\pi_{e'}\cap C\not=\emptyset$, nor does it intersect any of the finitely many $\gamma_{e'}$ that were defined in some step before $n$ (there are finitely many such $\gamma_{e'}$ that intersect $C$). By construction, the family $\{\gamma_e:\, e\in E(G)\setminus E(T)\}$ has properties (1) and (2).

This family can be used to construct a fiid decoration of $G$ by a tiling that represents $G$, as we have explained above.
To finish, we want to turn this into an $\Isom (\R^d)$-invariant tiling that represents $G$.
First, we can look at the fiid decoration of $V(T)=V(G)$ by the Cayley diagram $\Z^d$, as in Theorem \ref{Tdomains}. On top of this decoration of $G$ by a copy of $\Z^d$, consider the further decoration by a tiling that represents $G$ as constructed above. Apply the Duality lemma (Lemma \ref{dual}), to obtain a decoration of $\Z^d$ which is $\Aut (\Z^d)$-invariant and represents $G$ by tiles. Now take the expansion $\R^d$ of $\Z^d$, which is the space where the tiles are sitting in, and pick a uniform random element $g\in \Isom (\R^d)/\Aut (\Z^d)$. Move $\Z^d$ and the tiles by $g$ in $\R^d$.  The resulting random decorated copy of $\Z^d$ in $\R^d$ is $\Isom (\R^d)$-invariant.
\end{proof}

\begin{proof}[Proof of Theorem \ref{mackovilag}]
Follows as a special case of Theorem \ref{Gdomains}.
\end{proof}

\section{Representing $T_3$ by indistinguishable tiles}\label{s.indisttiles}

\begin{proof}[Proof of Theorem \ref{question}]
Let $G$ be the Cayley diagram of BS(1,2) in the representation $<a,b|a^{-1}ba=b^2>$. 
By Theorem \ref{Gdomains},
the graph underlying $G$ can be represented in $\R^d$, $d\geq 3$, by a
tiling that is $\Isom (\R^d)$ invariant.

Let $\calp$ be the (deterministic) subgraph consisting of edges of $G$ that are colored by $b$ in the Cayley diagram of $G$. The infinite clusters (that we will call fibers) are biinfinite paths, and they are (deterministically) indistinguishable with scenery $(G,\calp)$, because any of them can be taken to any other by an automorphism that preserves the fibers. For each fiber, take the union of the tiles that decorate the vertices in that fiber. This union is a connected infinite tile, because the tiles corresponding to a fiber form a connected set with regard to adjacency of tiles. 
The fibers with their original decorations with tiles are indistinguishable by the Decoration Lemma (Lemma \ref{lemma_fiid}). Hence the unions of the tiles over the fibers are also indistinguishable.

We have concluded that the pieces are indistinguishable. Their adjacency graph is $T_3 $ by definition. This finishes the proof.
\end{proof}

\begin{remark}\label{dim2}{\bf (No $T_3$-tiling in $\R^2$)}
There is no representation of $T_3$ in $\R^2$ by an ergodic random tiling of indistinguishable tiles, as shown by the following proof by contradiction. Suppose there is such a partition, denote by $D_0$ the tile of the origin $o$, and by $D_1,D_2,D_3$ its neighbors. There exists an $r>0$ such that with positive probability there exists a ball $B$ of radius $r$ with center in $D_0$, such that $B$ intersects each of $D_1,D_2$ and $D_3$ (call such balls {\it $D_0$-trifurcating}). Suppose that with positive probability there exist three $D_0$-trifurcating balls $B_1,B_2,B_3$ that are pairwise disjoint. Let the center of $B_i$ be $o_i$, and denote some arbitrarily chosen intersection point of the $i$'th ball with $D_j$ by $x_i^j$, for $j=1,2,3$. There exist paths $P_i^j$ between $o_i$ and $x_i^j$ ($j=1,2,3$) inside $B_i$ whose interiors are pairwise disjoint. Also, all these 9 paths have pairwise disjoint interiors, because the $B_i$ are pairwise disjoint. Furthermore, there is a path $P_j$ within $D_j$ that contains $x_1^j$, $x_2^j$ and $x_3^j$, and such that $P_j\setminus\{x_1^j,x_2^j,x_3^j
\}$ is disjoint from all the previously defined paths. The union $\bigcup_{i,j} P_i^j\cup P_j$ can be regarded as a graph on 12 vertices embedded in $\R^2$. This graph contains $K_{3,3}$ as a minor (contract the paths $P_j$ to one vertex, as $j=1,2,3$), contradicting the fact that this graph is embedded in the plane. This contradiction shows that the probability of having three pairwise disjoint $D_0$-trifurcating balls $B_1,B_2,B_3$ is zero. Therefore, all $D_0$-trifurcating balls have to be within 
the $r$-neighborhoods of at most two $r$-balls a.s.. Call a ball of radius $r$ a {\it trifurcating} ball if its center is in a tile $ D$ and it intersects all the neighboring tiles of $ D$. 
We have seen that the probability of having three pairwise disjoint $D_0$-trifurcating balls is 0. Consequently, since there are countably many tiles $D$ in the tiling, the probability that any of them has 3 pairwise disjoint trifurcating balls is 0.
We know that with positive probability there do exist trifurcating balls. Define a mass transport $(8r^2\pi)^{-1}$
from every point of $D$ to every point of $D$ that is in a trifurcating ball. The expected mass sent out is less than 1. (Here we are applying the continuous version of the MTP, as defined in \cite{BS}.) 
The expected mass received is equal to $\P(o$ is in a trifurcating ball$)\E(\Area(D_0)|o$ is in a trifurcating ball$)$. 
Therefore, conditioned on that $o$ is in a trifurcating ball, $D_0$ has finite area. By ergodicity and indistinguishability, then all tiles have the same, finite area $c$ almost surely. Choose a uniform point in each of the tiles, let the set of all these chosen points be $X$, and consider the graph $T_3$ on $X$ that they inherit from the tiling. By Proposition \ref{nincsfa} such an invariant representation of $T_3$ in $\R^2$ is not possible if the set of vertices has finite intensity as a point process. But then, using invariance, a unit square in $\R^2$ contains infinitely many elements of $X$ with positive probability. Hence the expected number of points of $X$ in this square is infinite, contradicting the fact that a tile of area $c$ contains 1 point of $X$ (which implies by the MTP that a unit cube contains $1/c$ points in expectation). 
\end{remark}

With a simple extra trick, we can prove Theorem \ref{amenindist} similarly to Theorem \ref{question}. 

\begin{proof}[Begin proof of Theorem \ref{amenindist}]
Consider a Cayley diagram that defines the given Cayley graph and take its Cartesian product with a Cayley diagram of $\Z$. The resulting amenable Cayley diagram has a representation by an invariant tiling by Theorem \ref{mackovilag}. Now call the copies of $\Z$ in $G\times \Z$ {\it fibers}, and define a new tiling by taking the union of tiles over every fiber. By the same argument as in the proof of Theorem \ref{question}, we can check that the resulting tiling has indistinguishable pieces. By construction, it will represent $G$.
\end{proof}

\section{Questions and further directions}
An alternative, discrete formulation of the original question remains open:

\begin{question}\label{question1} (Itai Benjamini)
Is there an invariant percolation on $\Z^3$ such that connected components are indistinguishable, and their adjacency graph is the 3-regular tree $T_3$?
\end{question}

Several follow-up questions can be asked about the tiling in our construction, which were raised by Benjamini:

\begin{question}
Consider some invariant random tiling of $\R^d$ ($d\geq 3$) that represents $T_3$. What can we say about the volume growth of the tiles? Can the tiles be convex? 
\end{question}

\begin{remark}
One of these questions was whether there is a ``tree factor'' in $\R^2  \times \R/\Z$. The answer seems to be positive. As long as the space is amenable and BS(0,1) can be invariantly embedded there, our proof seems to work.
\end{remark}

In light of Theorems \ref{question} and \ref{amenindist}, it is natural to ask whether any Cayley graph (or unimodular random graph) $G$ can be represented by indistinguishable tiles. Our proof gives an affirmative answer when $G$ can be obtained by contracting the components of some automorphism-invariant percolation on some amenable transitive graph (diagram).

\begin{question}\label{nagykerdes}
Let $G$ be a transitive graph and $d\geq 3$. When is there an $\Isom (\R^d)$-invariant random tiling representation of $G$ in $\R^d$ with indistinguishable tiles?
\end{question}

A tempting way to tackle the question for a Cayley graph $G$ would be through the following argument. For simplicity, suppose that $G$ is generated by 2 elements, and let $N$ be the normal subgroup such that $F_2/N$ is isomorphic to $G$. A representation of the 4-regular tree $T_4$ by an invariant tiling ${\cal T}$ is possible, similarly to Theorem \ref{question}. By turning $T_4$ into a Cayley diagram of $F_2$, every coset $Ng$ of $N$ will define a set ${\cal C}_g=\cup_{x\in Ng}   T(x)$ where $T(x)\in {\cal T}$ is the tile representing the group element $x$. With proper care the set $\{{\cal C}_g: \, g\in G\}$ will be invariant, and it looks like one could adapt our arguments to show that the pieces ${\cal C}_g$ are indistinguishable. However, the sets ${\cal C}_g$ are generally not connected. The cosets are nonamenable structures (unless $N$ is generated by a single element), hence
the method of Section \ref{s.treetiling} fails in this context, and
it is not clear how one could redefine the sets ${\cal C}_g$ in a consistent way that stabilizes in the limit.

\begin{remark}\label{mindjartvege}
Suppose that $G$ has a representation by an invariant random tiling with indistinguishable tiles in $\R^d$. For any two adjacent tiles $T$ and $T'$ in a given configuration, one can define a perfect matching $\beta_{T,T'}$ between $T$ and $T'$ (thought of as a map from $T$ to $T'$, such that $\beta_{T',T}=\beta^{-1}_{T,T'}$) which is equivariant with $(T,T')$.
Such a construction is possible with the further property that $\beta_{T,T'}$ preserves the Lebesgue measure $\lambda$, i.e., $\lambda(\beta_{T,T'} (A))=\lambda(A)$ for every measurable $A\subset T$. We omit the details here, but a version of the stable matching as in \cite{HHP} can be used to construct $\beta_{T,T'}$. One can make the perfect matching have some further nice properties, such as smoothness outside of a set of measure 0. Now, if $G$ is the left Cayley diagram of a group $\Gamma$, generated by a finite set $S$, and a tiling representation of $G$ as above is given, let $T_o$ be the tile containing the origin of $\R^d$, and $T_v$ be the tile that represents $vo\in\Gamma= V(G)$. Define an action of $\Gamma$ on $\R^d$ as follows. For an $x\in \R^d$ and $g\in S$ define $g(x)= \beta_{T_v,T_{gv}}(x)$ where $x\in T_v$. Hence, to every instance of the random tiling there corresponds a Lebesgue measure preserving action of $\Gamma$ on $\R^d$. Now, let $\Delta$ be the set of all Lebesgue measure preserving bijections up to measure 0 of $\R^d$ to itself. (Alternatively, we could choose $\Delta$ to be the set of piecewise smooth maps, for example.) Then we have just constructed a random subgroup $\Delta'\leq \Delta$ such that $\Delta'$ is isomorphic to $\Gamma$ almost surely. By the $\Isom(\R^d)$-invariance of the tiling and the equivariance of the $\beta_{T,T'}$, $\Delta'$ is invariant under conjugation by $\Isom (\R^d)\leq \Delta$. Such ``partially invariant random subgroups'' might be of interest. See \cite{AGV} for a seminal work on invariant random subgroups.
\end{remark}

\begin{question}\label{kiskerdes}
Let $\Delta$ be the set of all Lebesgue measure preserving bijections of $\R^d$ to itself.
Is there a finitely generated group $G$ such that $\Delta$ has no random subgroup isomorphic to $G$ and invariant under conjugation by $\Isom (\R^d)$?
\end{question}
An example when there is no tiling in Question \ref{nagykerdes} would imply a positive answer to Question \ref{kiskerdes}, by an argument as in Remark \ref{mindjartvege}. Since we do have such partially invariant random subgroups when $G$ is a free group or an amenable group, Kazhdan groups may be the first candidates to find an example for Question \ref{kiskerdes}.

%%%%%%%%%%%%%%%%%%%%%%%%%%%%%%%%%%%%%%%%%%%%%%
%% Single Appendix:                         %%
%%%%%%%%%%%%%%%%%%%%%%%%%%%%%%%%%%%%%%%%%%%%%%
%\begin{appendix}
%\section*{???}%% if no title is needed, leave empty \section*{}.
%\end{appendix}
%%%%%%%%%%%%%%%%%%%%%%%%%%%%%%%%%%%%%%%%%%%%%%
%% Multiple Appendixes:                     %%
%%%%%%%%%%%%%%%%%%%%%%%%%%%%%%%%%%%%%%%%%%%%%%
%\begin{appendix}
%\section{???}
%
%\section{???}
%
%\end{appendix}

%%%%%%%%%%%%%%%%%%%%%%%%%%%%%%%%%%%%%%%%%%%%%%
%% Support information (funding), if any,   %%
%% should be provided in the                %%
%% Acknowledgements section.                %%
%%%%%%%%%%%%%%%%%%%%%%%%%%%%%%%%%%%%%%%%%%%%%%
\section{Acknowledgements}
This work was started at the Bernoulli Center (CIB) conference ``Statistical physics on transitive graphs''. I would like to thank Itai Benjamini, Dorottya Beringer, Damien Gaboriau, Russ Lyons, G\'abor Pete, Mikael De La Salle and Romain Tessera for inspiring conversations, and a referee for useful comments.

The author was supported by a Marie Curie Intra European Fellowship within the 7th European Community Framework Programme, by the Hungarian National Research, Development and Innovation Office, NKFIH grant K109684, and by grant LP 2016-5 of the Hungarian Academy of Sciences.
% 
% The first author was supported by ...
% 
% The second author was supported in part by ...

%%  MR numbers will be added by VTeX.                      %%
%%                                                         %%
%%  Use \cite{...} to cite references in text.             %%
%%                                                         %%
%%%%%%%%%%%%%%%%%%%%%%%%%%%%%%%%%%%%%%%%%%%%%%%%%%%%%%%%%%%%%

%% if your bibliography is in bibtex format, uncomment commands:
%\bibliographystyle{imsart-number} % Style BST file (imsart-number.bst or imsart-nameyear.bst)
%\bibliography{bibliography}       % Bibliography file (usually '*.bib')

\begin{thebibliography}{AAA}

\bibitem{AGV}
Ab\'ert, M., Glasner, Y., Vir\'ag, B.
(2014) Kesten's theorem for Invariant Random Subgroups
{\it Duke Math. J.} 163, no. 3, 465-488.


\bibitem{AL}
Aldous, D., Lyons, R. (2007) 
Processes on unimodular random networks
{\it Electron. J. Probab.}, {\bf 12}, 1454-1508.


\bibitem{AHNR}
O. Angel, T. Hutchcroft, A. Nachmias and G.Ray,
Hyperbolic and parabolic unimodular random graphs. 
{\it Geom. Funct. Anal.} 28
(2018)  879--942.

%\bibitem{BK}
%Burton, R. M., Keane, M. (1991)
%Topological and metric properties of infinite clusters in stationary two-dimensional site percolation. 
%(English summary) 
%{\it Israel J. Math.} {\bf 76}, no. 3, 299-316. 

\bibitem{BLPS}
Benjamini, I., Lyons, R., Peres, Y., Schramm, O. (1999) Group-invariant
percolation on graphs {\it Geom. Funct. Anal.} {\bf 9}, 29-66.

\bibitem{BS}
Benjamini, I., Schramm, O. (2001) Percolation in the hyperbolic plane. {\it J. Amer. Math.
Soc.}, 14(2):487-507 (electronic).

\bibitem{BT}
Benjamini, I., Tim\'ar, \'A. (2019) Invariant embeddings of unimodular random planar graphs (preprint). {\tt arXiv:1910.01614}

%\bibitem{BT}
%Beringer, D., Tim\'ar, \'A. (2017) Duality of unimodular random decorations of unimodular networks (work in progress).

\bibitem{G}
Gaboriau, D (1998) Mercuriale de groupes et de relations. {\it C. R. Acad. Sci. Paris S\'er. I Math.},
326(2):219–222.


\bibitem{HHP}
Hoffman, C., Holroyd, A., Peres, Y. (2006) A stable marriage of Poisson and Lebesgue
{\it Ann. Probab.}
{\bf 34}, Number 4 (2006), 1241-1272.
%\bibitem{LN} Lyons, R., Nazarov, F. (2011) Perfect matchings as IID factors on non-amenable groups, Europ. J. Combin. Vol. 32, 1115-1125.

%\bibitem{LP} Lyons, R., Peres, Y. (2015) {\em Probability on Trees and Networks}.\newblock Cambridge University Press.\newblock In preparation. Current version available at {\tt http://pages.iu.edu/\string~rdlyons/}.


\bibitem{LS}
Lyons, R., Schramm, O. (1999) Indistinguishability of percolation clusters {\it Ann. Probab.}
{\bf 27}, no. 4, 1809-1836.

\bibitem{M}
Martineau, S. (2015) Ergodicity and indistinguishability
in percolation theory {\it Enseignement math\'ematique}, {\bf 61}, 285-320.

\bibitem{OW}
Ornstein, D.S., Weiss, B.
(1980)
Ergodic theory of amenable group actions. I: The Rohlin lemma 
{\it Bull. Amer. Math. Soc. (N.S.)}
Volume 2, Number 1, 161-164.

\bibitem{P}
Pete, G. (2015) 
\emph{Probability and geometry on groups.} 
Lecture notes for a graduate course, 
Version of 3 August 2015. \,
{\tt http://www.math.bme.hu/$\sim$gabor/PGG.pdf}

\bibitem{T}
Tim\'ar, \'A. (2004) Tree and Grid Factors of General Point Processes, {\it Electronic
Communications in
Probability}
{\bf 9}, 53-59.

\bibitem{T2}
Tim\'ar, \'A. (2018) Invariant tilings and unimodular decorations of Cayley graphs, in {\it Unimodularity in Randomly Generated Graphs, Contemporary Mathematics} {\bf 719}, 43-61.

\bibitem{T3}
Tim\'ar, \'A. (2019) One-ended spanning trees in amenable unimodular graphs, {\it Electronic Communications in Probability} {\bf 24}, paper no. 72, 12 pp.


\end{thebibliography}

%% or include bibliography directly:
% \begin{thebibliography}{}
% \bibitem{b1}
% \end{thebibliography}

\noindent
\ \\
{
Alfr\'ed R\'enyi Institute of Mathematics\\
Re\'altanoda u. 13-15, Budapest 1053 Hungary\\
and Division of Mathematics, The Science Institute, University of Iceland\\
Dunhaga 3 IS-107 Reykjavik, Iceland.\\}

\end{document}